\documentclass[10pt]{article}
\usepackage[top=1.25in, bottom=1.5in, left=1.25in, right=1.25in]{geometry}
\usepackage{fancyhdr}
\usepackage{amsfonts}
\usepackage{amsthm}
\usepackage{amsmath}
\usepackage{amssymb}
\usepackage{stackrel}
\usepackage{latexsym}
\usepackage{amsfonts}
\usepackage{setspace}
\usepackage{multicol}
\usepackage{float}
\usepackage{graphicx}
\usepackage{multirow}
\usepackage{enumitem}
\usepackage{lineno}
\usepackage{verbatim}
\usepackage[mathscr]{euscript}
\usepackage{scalerel}[2016/12/29]

\usepackage[pdftex,pagebackref,hypertexnames=false, colorlinks, citecolor=black, linkcolor=blue, urlcolor=red]{hyperref}
\usepackage{listings}
\usepackage{color}
\usepackage{bm}
\usepackage{tikz}
\usetikzlibrary{arrows, positioning,chains,fit,shapes,calc}
\usetikzlibrary{decorations.pathmorphing}
\tikzset{
  text style/.style={
    sloped, 
    text=black
  }}
\usetikzlibrary{arrows,chains,matrix,positioning,scopes}
\usepackage{tikz-cd}
\usepackage{graphicx, multicol}
\usepackage{caption}
\usepackage{stmaryrd}
\usepackage{mathdots}

\usepackage{chngcntr}
\counterwithin{figure}{section}
\makeatletter
\tikzset{join/.code=\tikzset{after node path={%
\ifx\tikzchainprevious\pgfutil@empty\else(\tikzchainprevious)%
edge[every join]#1(\tikzchaincurrent)\fi}}}
\makeatother
\tikzset{>=stealth',every on chain/.append style={join},
         every join/.style={->}}
\tikzstyle{labeled}=[execute at begin node=$\scriptstyle,
   execute at end node=$]
   
\makeatletter
\newcommand*{\encircled}[1]{\relax\ifmmode\mathpalette\@encircled@math{#1}\else\@encircled{#1}\fi}
\newcommand*{\@encircled@math}[2]{\@encircled{$\m@th#1#2$}}
\newcommand*{\@encircled}[1]{%
  \tikz[baseline,anchor=base]{\node[draw,circle,outer sep=0pt,inner sep=.2ex] {#1};}}
\makeatother

\newcommand*{\defeq}{\mathrel{\vcenter{\baselineskip0.5ex \lineskiplimit0pt
            \hbox{\scriptsize.}\hbox{\scriptsize.}}}%
                     =}

\makeatletter
\newcommand*{\sumcirc}{%
  \DOTSB
  \mathop{
    \mathchoice
      {\rlap{\kern.2em\rotatebox[origin=d]{-90}{$\scaleobj{.9}{\varbigcirc}$}}{\sum}}
      {\vcenter{\rlap{\kern.15em\rotatebox[origin=d]{180}{$\scriptscriptstyle\circlearrowleft$}}}{\sum}}
      {\sum}{\sum}
  }\slimits@
}
\makeatother

\DeclareMathOperator{\supp}{supp}

\newtheorem{theorem}{Theorem}[section]

\newtheorem{prop}[theorem]{Proposition}
\newtheorem{corollary}[theorem]{Corollary}
\theoremstyle{definition}
\newtheorem{defn}[theorem]{Definition}
\newtheorem{example}[theorem]{Example}
\newtheorem{rmk}[theorem]{Remark}

\numberwithin{equation}{section}

\begin{document}
\title{When do two networks have the same steady-state ideal? }

\author{
Mark Curiel \\
University of Hawai'i at M\={a}noa
\and 
Elizabeth Gross \\
University of Hawai'i at M\={a}noa
\and
Carlos Mu\~{n}oz \\
San Jos\'{e} State University}
\maketitle

\abstract{Chemical reaction networks are often used to model and understand biological processes such as cell signaling. Under the framework of chemical reaction network theory, a process is modeled with a directed graph and a choice of kinetics, which together give rise to a dynamical system.  Under the assumption of mass action kinetics, the dynamical system is polynomial. In this paper, we consider the ideals generated by the these polynomials, which are called steady-state ideals. Steady-state ideals appear in multiple contexts within
the chemical reaction network literature, however they have yet to be systematically studied. To begin such a study, we ask and partially answer the following question:  when do two reaction networks give rise to the same steady-state ideal? In particular, our main results describe three operations on the reaction graph that preserve the steady-state ideal. Furthermore, since  the motivation for this work is the classification of steady-state ideals, monomials play a primary role.  To this end, combinatorial conditions are given to identify monomials in a steady-state ideal, and we give a sufficient condition for a steady-state ideal to be monomial.}

\section{Introduction}

Under the assumption of mass-action kinetics, every chemical reaction network gives rise to a polynomial dynamical system that governs the rate of change of the participant chemical species concentrations. Such reaction networks have significance and applications in chemistry and biochemistry. For instance, reaction network models can be used to study the effect of a biological catalyst on the production or destruction of cells as in the Wnt signaling pathway model \cite{maclean2015parameter} \cite{GHRS}. Because of their impact in other fields, the mathematics of chemical reaction networks have been studied over the decades to better understand their behavior and their strengths and limitations in modeling real-world phenomena (see e.g. \cite{horn1972}, \cite{feinberg1972}).  In this paper, we will take an algebraic approach, along the lines of \cite{DickensteinSurvey}, and consider the polynomials defining these systems and their corresponding ideals. These ideals are known as the \emph{steady-state ideals}. 

Steady-state ideals are algebraic objects that appear in multiple contexts within the chemical reaction network literature. For example, elements of the steady-state ideal are relationships among species concentrations that must hold at steady-state, and these polynomials, or \emph{steady-state invariants}, can be used for model selection \cite{Gunawardena}\cite{HarringtonHo}\cite{modelselection}.  Such applications have motivated the study of algebraic operations on steady-state ideals, such as elimination \cite{Joining}. As another example, steady-state ideals can be used for the design of experiments as illustrated in \cite{maclean2015parameter}.  In \cite{maclean2015parameter}, the authors analyze the matroids defined by the associated primes of the steady-state ideal to determine optimal sets of species concentrations to measure.

Fundamentally, the steady-state ideal contains information about the system at steady-state. Given rate parameters and initial conditions, equilibria can be found by intersecting the variety defined by the steady-state ideal with the appropriate stochiometric compatability class,  an affine space governed by the reactions in the network and the initial conditions. Thus, information about the steady-state ideal can give us insight into the existence of positive steady-states or multistationarity.  For instance, the authors of \cite{MDSC} give a condition for multistationarity that holds when the steady-state ideal is generated by binomials.  Indeed, understanding the generators of the steady-state ideal can be quite informative. Information about the steady-states is also gained when the steady-state ideal has a  monomial generator. In particular, a monomial in the steady-state ideal signifies that all steady-states lie on the boundary of the positive orthant and, thus the system has no positive steady-states. 

There are several known results that implicitly or explicitly give information about the steady-state ideal from the combinatorics of the underlying reaction network. Returning to the example of binomial ideals, the authors of \cite{toricdynamicalsystems} show that if a reaction network is weakly reversible and has deficiency zero then its steady-state ideal is generated by binomials.   However, while steady-state ideals are of interest, there are still only a few works aimed at classifying steady-state ideals.  This paper aims to provide a basis for the classification problem by examining the  title question, \emph{When do two networks have the same steady-state ideal?}, for a specific class of networks that we call $0,1$-networks, i.e. networks whose complexes all lie in $\{0,1\}^{\mathcal S}$.  To this end, our main theorems give three ideal preserving network operations (Theorem \ref{addproduct}, Theorem \ref{addreactant}, and Theorem \ref{addreaction}).  Specifically, these network operations add species to reactant complexes, add species to product complexes, and add reactions--- and if certain combinatorial criteria are satisfied, then the steady-state ideals remain unchanged. Since  the motivation for this work is the classification of steady-state ideals, monomials play a primary role.  In particular, combinatorial conditions are given to identify monomials in a steady-state ideal (Proposition \ref{almostbalanced} and Proposition \ref{prodalmostbalanced}), and we give a sufficient condition for a steady-state ideal to be monomial (Theorem \ref{min_react_generate_ideal}).

This paper is organized as follows. Section 2 introduces the mathematics of chemical reaction networks. In particular, we describe how to obtain the steady state ideal from a reaction network under the assumption of mass action kinetics. Moreover, we introduce the reader to the relevance of the graphical structure of the network by highlighting the connection between the monomial support of the polynomials in the steady-state ideal and vertices of the reaction network with outgoing edges. In Section 3, we introduce almost balanced vertices. These vertices belong to the network hypergraph defined by the reaction network. A network hypergraph can be defined for any network, however, we detail the consequences of almost balanced vertices in a family of  \textit{0,1-networks}. Almost balanced vertices are related to the idea of balanced edges sets of hypergraphs that arise in algebraic statistics \cite{PS} \cite{gross2013combinatorial}. In this section, we show that almost balanced vertices are combinatorial signatures of monomials in the steady-state ideal.
In addition, we show that a hypergraph arsing from a weakly reversible network, that is, a network with a strongly connected underlying graph, cannot have almost balanced vertices. This is consistent with the literature as weak reversibility is a sufficient condition on networks with the capacity to exhibit positive steady states \cite{boros}. In Section 4, we discuss operations on reaction networks, and we introduce three operations on 0,1-networks that preserve the steady-state ideal, the main theorems in this paper. Finally, in the Conclusion, we discuss ideas for moving beyond networks with monomials in their steady-state ideal and 0,1-networks.

\section{Background}

\begin{defn} A \textbf{chemical reaction network} $N$ is a triple $(\mathscr{S},\mathscr{C},\mathscr{R})$ where $\mathscr{S} = \{s_1, \ldots, s_n\}$ is a set of chemical species, $\mathscr{C} = \{y_1, \ldots, y_m\} \subseteq \mathbb{Z}_{\geq0}^\mathscr{S}$ is a set of chemical complexes, and $\mathscr{R}$, a set of chemical reactions, is a relation on $\mathscr{C}$ denoted by $y_i \to y_j$. Moreover, we require that a chemical reaction network $N$ satisfies the following: $y_i \to y_i \not \in \mathscr{R}$ for all $y_i \in \mathscr{C}$; if $y_i \in \mathscr{C}$, then there is a $y_j \in \mathscr{C}$ such that either $y_i \to y_j \in \mathscr{R}$ or $y_j \to y_i \in \mathscr{R}$; and $\mathscr{S} = \bigcup_{y_i \in \mathscr{C}} \text{supp}(y_i)$ where $\supp(y_i)$ is the set $\{s_j \in \mathscr{S} \mid y_{i_j} \neq 0\}$.   Due to algebraic reasons, we will also assume that $\mathscr R$ does not contain any reactions of the form $\varnothing \to y$ (production), but we will allow reactions of the form $ y \to \varnothing$ (degradation).
\end{defn}

\begin{rmk}
Throughout this paper, complexes in a reaction will commonly be given the same index as the respective reaction, that is, if $y_i \to y_j$ is the $k$th reaction, then this reaction will be denoted as $y_k \to y_k'$ where $y_k = y_i$ and $y_k' = y_j$.
\end{rmk}

 To each complex $y_i = (y_{i_{s_1}}, y_{i_{s_2}}, \ldots, y_{i_{s_n}})\in \mathscr{C}$, we associate the monomial $x^{y_i} = \prod_{s_j \in \mathscr{S}} x_{s_j}^{y_{i_{s_j}}}$ where $x_{s_j}$ is the concentration of species $s_j$.  The \textit{law of mass action kinetics} states that the change in concentration of a species is proportional to the product of concentrations of reacting species. Thus, under this assumption, we obtain a system of differential equations given by:
\begin{equation}\label{eq:sseq}
    \dot{x}_{s_j} = \sum_{y_i \to y_i' \in \mathscr{R}} \kappa_{i} x^{y_i} (y_{i_{s_j}}' - y_{i_{s_j}}) \;,
\end{equation}
where $\kappa_{i}$ is the proportional rate constant for $y_i \to y_i'$. For us, we will view the differential equation $\dot{x}_{s}$ as a polynomial in the ring $\mathbb{K}(\bm{\kappa})[\bm{x}]$ and, for each $s \in \mathscr{S}$, we will call  $\dot{x}_s$ the \textit{steady-state polynomial} for $s$. Chemical reaction networks where $y_{i_s} \in \{0,1\}$ for all $i$ and $s$ will be called \textit{0,1-networks}.  The quantity $y_{i_s}' - y_{i_s}$ from Equation \ref{eq:sseq} represents the net change of species $s$ in the reaction $y_i \to y_i'$ and will be denoted $\gamma_{s_i}$.

\begin{example}
Let $N$ be the  chemical reaction network in Figure \ref{examplenetwork} with three complexes.  There are two species $\mathscr S =\{A, B\}$ and three complexes $2A$, $2B$, and $A+B$ that we view as vectors in $\mathbb Z_{\geq 0}^{\mathscr S}$, i.e. $\mathscr C = \{ (2,0), (0,2), (1,1) \}$.  This network is not a 0,1-network since not all vectors in $\mathscr C$ are 0,1 vectors.  The three reactions are represented as edges in the directed graph shown.  The mass action system is shown below to the right of the network.
\begin{figure}[H]
    \centering
    \begin{tikzcd}[/tikz/column 1/.append style={anchor=base east}
    ,/tikz/column 2/.append style={anchor=base west}
    ]
    2A \ar[rd, "\kappa_{1}"]  \ar[rr, bend left = 30, "\kappa_{2}"] & & 2B & \dot{x}_A = -\kappa_{1}x_A^2 - 2\kappa_{2}x_A^2 - \kappa_{3}x_Ax_B\\
    & A + B \ar[ur ,"\kappa_{3}"] &&\hspace{-0.25cm} \dot{x}_B = \kappa_{1}x_A^2 + 2\kappa_{2}x_A^2 + \kappa_{3}x_Ax_B
    \end{tikzcd}
    \caption{Chemical reaction network with two species and the associated polynomial dynamical system.}
    \label{examplenetwork}
\end{figure}
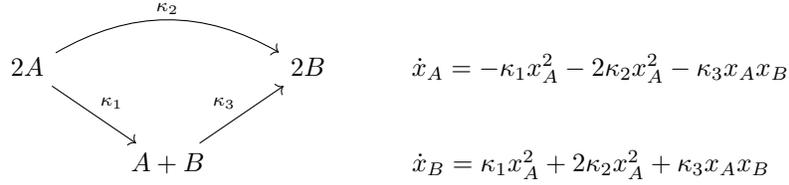
\end{example}

\begin{defn}
Given a chemical reaction network $N = (\mathscr{S},\mathscr{C},\mathscr{R})$, the \textbf{steady-state ideal} $\mathcal{I}(N) \subseteq \mathbb{K}(\bm{\kappa})[ {\bf x}]$ is the ideal $\langle \dot{x}_s : s \in \mathscr{S} \rangle$.
\end{defn}

\begin{example} The steady state ideal of the network in Figure \ref{examplenetwork} is
\begin{align*}
    \mathcal{I}(N) 
    &= \langle \dot{x}_A, \dot{x}_B \rangle \\
    &= \langle -\kappa_{1}x_A^2 - 2\kappa_{2}x_A^2 - \kappa_{3}x_Ax_B, \kappa_{12}x_A^2 + 2\kappa_{13}x_A^2 + \kappa_{23}x_Ax_B \rangle \\
    &= \langle \kappa_{1}x_A^2 + 2\kappa_{2}x_A^2 + \kappa_{3}x_Ax_B \rangle.
\end{align*}
\end{example}

As the next example shows the steady-state ideals of two different networks can be generated by the same set of polynomials.  

\begin{example} 
Consider the following two 0,1 reaction networks on the same set of species
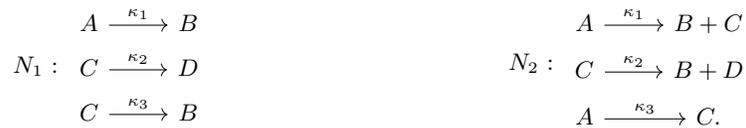
\begin{multicols}{2}
    \centering
    $N_1$ :
    \begin{tikzcd}[row sep = tiny]
    A  \ar[r,"\kappa_1"] & B\\
    C \ar[r,"\kappa_2"] & D \\
    C \ar[r,"\kappa_3"] & B \\
    \end{tikzcd}
    
    $N_2$ :
    \begin{tikzcd}[row sep = tiny]
    A \ar[r,"\kappa_1"] & B + C\\
    C \ar[r,"\kappa_2"] & B + D\\
    A \ar[r,"\kappa_3"] & C.
    \end{tikzcd}
\end{multicols}
The steady state ideal for $N_1$ is 
$\mathcal{I}(N_1) = \langle -\kappa_1  x_A, \kappa_1 x_A+\kappa_3x_C, -\kappa_2 x_C-\kappa_3x_C, \kappa_2 x_C \rangle =$ $\langle x_A, x_C \rangle \subseteq $ $\mathbb K( \kappa_1, \kappa_2, \kappa_3)[x_A, x_B, x_C, x_D]$ while the steady state ideal for $N_2 $ is $\mathcal{I}(N_2) = \langle -\kappa_1  x_A - \kappa_3 x_A,  \kappa_1 x_A + \kappa_2 x_C, \kappa_1 x_A - \kappa_2 x_C - \kappa_3x_C, \kappa_2 x_C \rangle\subseteq $ $\mathbb K( \kappa_1, \kappa_2, \kappa_3)[x_A, x_B, x_C, x_D]$.  Thus, the steady state ideals for $N_1$ and $N_2$ are equal.


Notice that the pair of networks $N_1$ and $N_2$ have the same species sets. A pair of networks can have different species sets but still have equal steady state ideals if we consider both as ideals of a large enough ring. For instance, consider the following network
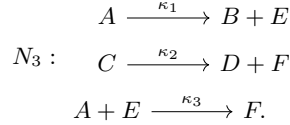
\[N_3:
\begin{tikzcd}[row sep = tiny]
A \ar[r,"\kappa_1"] & B + E\\
C \ar[r,"\kappa_2"] & D + F\\
A+E \ar[r,"\kappa_3"] & F.
\end{tikzcd}
\]
The steady-state ideal for $N_3$ is $\mathcal{I}(N_3) = \langle x_A,x_C \rangle \subseteq  \mathbb K( \kappa_1, \kappa_2, \kappa_3, )[x_A, x_B, x_C, x_D, x_E, x_F]$ and thus has the same generating set as the steady-state ideals corresponding to networks $N_1$ and $N_2$. If we consider the ideals $\mathcal{I}(N_1)$ and $\mathcal{I}(N_2)$ as ideals in $  \mathbb K( \kappa_1, \kappa_2, \kappa_3, )[x_A, x_B, x_C, x_D, x_E, x_F]$, then we have  $\mathcal{I}(N_3) = \langle x_A, x_C \rangle = \mathcal{I}(N_1) = \mathcal{I}(N_2)$.  In general, when comparing two steady-state ideals, we will consider them both as living in the same ring, in other words, we will call two steady-state ideals equal if they can both be generated by the same set of polynomials (see Remark \ref{rmk:equality}). An example of networks with the same steady state ideal under this definition are dynamically equivalent systems \cite{CJY}.

\end{example}

The following observation will be helpful in our exploration of steady-state ideals. The proof of Proposition 2.7 is straightforward, and we include it here as a way to introduce how we will reason about steady-state polynomials throughout the paper.

\begin{prop}\label{finestideal}
Reactants generate the support of all the polynomials in the steady-state ideal, i.e.

$$\mathcal{I}(N) \subseteq \langle x^{y_i} \ : \ y_i \to y_i' \in \mathscr{R} \rangle.$$
Moreover, if $\mathcal{I}(N)$ is monomial, then $\mathcal{I}(N) = \langle x^{y_i} \ : \ y_i \to y_i' \in \mathscr{R} \rangle$.
\end{prop}

\begin{proof} 
Let $f$ $\in \mathcal{I}(N)$. We know that $\mathcal{I}(N)$ is generated by each $\dot{x}_s$, where $s \in \mathscr{S}$. This means 
$$ f = \sum_{s \in \mathscr{S}} g_{s}  \dot{x}_{s}$$
where $g_{s} \in \mathbb K(\bm{\kappa})[{\bf x}]$.
Recall that $\dot{x}_s = \sum_{y_i \to y_i' \in \mathscr{R}} \gamma_{s_i} \kappa_{i} x^{y_i}$ Therefore, we can write $f$ in terms of $\sum_{y_i \to y_i' \in \mathscr{R}} \gamma_{s_i} \kappa_{i} x^{y_i}$. In particular,
$$ f = \sum_{s \in \mathscr{S}} g_{s} \left(\sum_{y_i \to y_i' \in \mathscr{R}} \gamma_{s_i} \kappa_{i} x^{y_i} \right).$$
Thus, after expanding, we can see every monomial in the expansion of $f$ is divisible by $x^{y_i}$ for some $i$ such that $y_i \to y_i' \in \mathscr{R}$. It remains to show that if $\mathcal{I}(N)$ is monomial, then $\mathcal{I}(N) = \langle x^{y_i} \ : \ y_i \to y_i' \in \mathscr{R} \rangle$.

Assume $\mathcal{I}(N)$ is monomial. We will show that $x^{y_i} \in \mathcal{I}(N)$ for each $y_i \in \mathscr{C}$ such that $y_i \to y_i' \in \mathscr{R}$. Consider the following equivalent notion of a monomial ideal, stated here in terms of the steady state ideal: for any polynomial $f$ belonging to $\mathcal{I}(N)$, each nonzero term of $f$ must also be an element of $\mathcal{I}(N)$. In particular, the nonzero terms of the generators $\dot{x}_s$ of $\mathcal{I}(N)$ are elements of $\mathcal{I}(N)$. Since the terms that make up $\dot{x}_s$ are real multiples of $x^{y_i}$ where $y_i \in \mathscr{C}$ such that $y_i \to y_i' \in \mathscr{R}$, then we need only show that the real multiples are nonzero. Recall, the coefficient of $x^{y_i}$ in $\dot{x}_s$ is $\gamma_{s_i}\kappa_i$. Since $\kappa_i$ is nonzero by definition, we will show $\gamma_{s_i}$ is nonzero. 

By the definition of a chemical reaction network, the reaction $y_i \to y_i$ is not in $\mathscr{R}$ for any $y_i \in \mathscr{C}$. Hence, for each reaction $y_i \to y_i' \in \mathscr{R}$, there is a species $s \in \supp{(y_i'-y_i)}$, since $y_i'-y_i$ is not the zero vector. Necessarily, $\gamma_{s_i}$ is nonzero and so the coefficient of $x^{y_i}$ in $\dot{x}_s$ is nonzero. Thus, since $\mathcal{I}(N)$ is monomial, then $x^{y_i} \in \mathcal{I}(N)$.
\end{proof}

\section{Combinatorics of Chemical Reaction Networks}\label{Combinatorics_of_CRNs}

We begin this section by defining the main combinatorial object that we will study in the paper: \emph{the network hypergraph}.

\begin{defn} Let $N = (\mathscr{S},\mathscr{C},\mathscr{R})$ be a chemical reaction network. For each $y_i \to y_i' \in \mathscr{R}$, let $u_{i} = y_i$ and $v_{i} = y_i'$. The \textit{network hypergraph} $\mathcal{H}_N$ for $N$ has vertex set $V = \{u_{i},v_{i} : y_i \to y_i' \in \mathscr{R}\}$ and edge set $E$ containing the hyperedges $E_s = \bigcup_{y_i \to y_i'} \{u_{i} : s \in \mathrm{supp}(y_i)\} \cup \{v_{i} : s \in \mathrm{supp}(y_i')\}$ for each $s \in \mathscr{S}$ and the hyperedges $E_i = \{u_{i},v_{i}\}$ for each $y_i \to y_i' \in \mathscr{R}$ if $y_i' \neq \varnothing$, else $E_{i} = \varnothing$ if $y_i' = \varnothing$.
\end{defn}

The hyperedge $E_s$ is called the \textit{species hyperedge} for $s \in \mathscr{S}$. The hyperedge $E_i$ is called the \textit{reaction hyperedge} for the $i$th reaction $y_i \to y_i' \in \mathscr{R}$. Note that, for 0,1-networks, we get the following equation:
\begin{equation}\label{speciesedgeequ}
    \dot{x}_s = \sum_{v_i \in E_s} \kappa_ix^{y_i} - \sum_{u_i \in E_s} \kappa_ix^{y_i}
\end{equation}

\begin{example} The network hypergraph of the network in Figure \ref{examplenetwork} is
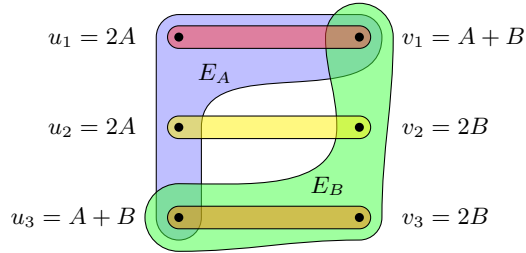
\[
\begin{tikzpicture}[scale = .6]
    \node (A1) at (-2,2) {};
    \node (A2) at (-2,0) {};
    \node (AB2) at (-2,-2) {};
    \node (AB1) at (2,2) {};
    \node (B1) at (2,0) {};
    \node (B2) at (2,-2) {};
    
    \begin{scope}[fill opacity=0.5]
    \filldraw[fill=blue!50] ($(AB1) + (0,.5)$)
        to[out=0,in=90] ($(AB1) + (.5,0)$)
        to[out=270,in=90] ($(A2) + (.5,0)$)
        to[out=270,in=90] ($(AB2) + (.5,0)$)
        to[out=270,in=0] ($(AB2) + (0,-.5)$)
        to[out=180,in=270] ($(AB2) + (-.5,0)$)
        to[out=90,in=270] ($(A1) + (-.5,0)$)
        to[out=90,in=180] ($(A1) + (0,.5)$)
        to[out=0,in=180] ($(AB1) + (0,.5)$);
    \filldraw[fill=green!75] ($(AB1) + (0,.75)$)
        to[out=0,in=90] ($(AB1) + (.75,0)$)
        to[out=270,in=90] ($(B2) + (.5,0)$)
        to[out=270,in=0] ($(B2) + (0,-.5)$)
        to[out=180,in=0] ($(AB2) + (0,-.75)$)
        to[out=180,in=270] ($(AB2) + (-.75,0)$)
        to[out=90,in=180] ($(AB2) + (0,.75)$)
        to[out=0,in=270] ($(B1) + (-.5,0)$)
        to[out=90,in=270] ($(AB1) + (-.75,0)$)
        to[out=90,in=180] ($(AB1) + (0,.75)$);
    \filldraw[fill=red!75] ($(A1) + (0,0.25)$)
        to[out=0,in=180] ($(AB1) + (0,0.25)$)
        to[out=0,in=90] ($(AB1) + (0.25,0)$)
        to[out=270,in=0] ($(AB1) + (0,-0.25)$)
        to[out=180,in=0] ($(A1) + (0,-0.25)$)
        to[out=180,in=270] ($(A1) + (-0.25,0)$)
        to[out=90,in=180] ($(A1) + (0,0.25)$);
    \filldraw[fill=yellow] ($(A2) + (0,0.25)$)
        to[out=0,in=180] ($(B1) + (0,0.25)$)
        to[out=0,in=90] ($(B1) + (0.25,0)$)
        to[out=270,in=0] ($(B1) + (0,-0.25)$)
        to[out=180,in=0] ($(A2) + (0,-0.25)$)
        to[out=180,in=270] ($(A2) + (-0.25,0)$)
        to[out=90,in=180] ($(A2) + (0,0.25)$);
    \filldraw[fill=orange] ($(AB2) + (0,0.25)$)
        to[out=0,in=180] ($(B2) + (0,0.25)$)
        to[out=0,in=90] ($(B2) + (0.25,0)$)
        to[out=270,in=0] ($(B2) + (0,-0.25)$)
        to[out=180,in=0] ($(AB2) + (0,-0.25)$)
        to[out=180,in=270] ($(AB2) + (-0.25,0)$)
        to[out=90,in=180] ($(AB2) + (0,0.25)$);
    \end{scope}
    
    \fill (A1) circle (0.1);
    \fill (A2) circle (0.1);
    \fill (AB1) circle (0.1);
    \fill (AB2) circle (0.1);
    \fill (B1) circle (0.1);
    \fill (B2) circle (0.1);
    
    \fill (-2.75,2) node [left] {$u_1=2A$};
    \fill (-2.75,0) node [left] {$u_2=2A$};
    \fill (2.75,2) node [right] {$v_1=A+B$};
    \fill (-2.75,-2) node [left] {$u_3=A+B$};
    \fill (2.75,0) node [right] {$v_2=2B$};
    \fill (2.75,-2) node [right] {$v_3=2B$};
    \fill (-1.2,1.2) node {$\scaleobj{1}{E_{A}}$};
    \fill (1.3,-1.3) node {$\scaleobj{1}{E_{B}}$};
\end{tikzpicture}
\]
The network hypergraph of the 0,1-network 
$\begin{tikzcd}[every arrow/.append style={shift left}]
A \ar[r,"\kappa_1"] & B \ar[l,"\kappa_2"]    
\end{tikzcd}$ is the following
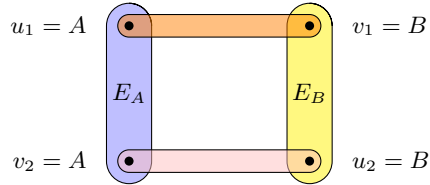
\[
\begin{tikzpicture}[scale = .6]
    \node (A1) at (-2,2) {};
    \node (A2) at (-2,-1) {};
    \node (B1) at (2,2) {};
    \node (B2) at (2,-1) {};
    
    \begin{scope}[fill opacity=0.5]
    \filldraw[fill=blue!50] ($(A1) + (0,.5)$) 
        to[out=0,in=90] ($(A1) + (.5,0)$)
        to[out=270,in=90] ($(A2) + (.5,0)$)
        to[out=270,in=0] ($(A2) + (0,-.5)$)
        to[out=180,in=270] ($(A2) + (-.5,0)$)
        to[out=90,in=270] ($(A1) + (-.5,0)$)
        to[out=90,in=180] ($(A1) + (0,.5)$)
        to[out=0,in=90] ($(A1) + (.5,0)$);
    \filldraw[fill=yellow] ($(B1) + (0,.5)$) 
        to[out=0,in=90] ($(B1) + (.5,0)$)
        to[out=270,in=90] ($(B2) + (.5,0)$)
        to[out=270,in=0] ($(B2) + (0,-.5)$)
        to[out=180,in=270] ($(B2) + (-.5,0)$)
        to[out=90,in=270] ($(B1) + (-.5,0)$)
        to[out=90,in=180] ($(B1) + (0,.5)$)
        to[out=0,in=90] ($(B1) + (.5,0)$);
    \filldraw[fill=orange] ($(A1) + (0,0.25)$)
        to[out=0,in=180] ($(B1) + (0,0.25)$)
        to[out=0,in=90] ($(B1) + (0.25,0)$)
        to[out=270,in=0] ($(B1) + (0,-0.25)$)
        to[out=180,in=0] ($(A1) + (0,-0.25)$)
        to[out=180,in=270] ($(A1) + (-0.25,0)$)
        to[out=90,in=180] ($(A1) + (0,0.25)$);
    \filldraw[fill=pink] ($(A2) + (0,0.25)$)
        to[out=0,in=180] ($(B2) + (0,0.25)$)
        to[out=0,in=90] ($(B2) + (0.25,0)$)
        to[out=270,in=0] ($(B2) + (0,-0.25)$)
        to[out=180,in=0] ($(A2) + (0,-0.25)$)
        to[out=180,in=270] ($(A2) + (-0.25,0)$)
        to[out=90,in=180] ($(A2) + (0,0.25)$);
    \end{scope}
    
    \fill (A1) circle (0.1);
    \fill (A2) circle (0.1);
    \fill (B1) circle (0.1);
    \fill (B2) circle (0.1);
    
    \fill ($(A1)+(-.75,0)$) node [left] {$u_1=A$};
    \fill ($(A2)+(-.75,0)$) node [left] {$v_2=A$};
    \fill ($(B1)+(.75,0)$) node [right] {$v_1=B$};
    \fill ($(B2)+(.75,0)$) node [right] {$u_2=B$};
    \fill ($(A2) + (0,1.5)$) node {$E_A$};
    \fill ($(B2) + (0,1.5)$) node {$E_B$};
\end{tikzpicture}
\]
\end{example}

For a majority of what follows, we will use multisets to capture combinatorial information from a network hypergraph. Formally, a multiset if defined as a pair $\mathscr{E} = (A,\mu)$ where $A$ is the underlying set for $\mathscr{E}$ and $\mu$ is a map from $A$ to $\mathbb{Z}_{\geq0}$ that outputs the multiplicity $\mu(a)$ for each element $a \in A$ as an element of $\mathscr{E}$. Given a hypergraph $\mathcal{H} = (V,E)$, a multiset $\mathscr{E}$ whose underlying set is $E$ is said to be a \textit{multiset over the edges in $E$}. When an edge $E_i \in E$ has nonnegative multiplicity $\mu(E_i) = \mu_i$ in $\mathscr{E}$, we write $\mathscr{E} = \left\{E_1^{(\mu_1)}, \ldots , E_m^{(\mu_m)}\right\}$. If $\mathscr{A} = \left\{E_1^{(\mu_1)}, \ldots , E_m^{(\mu_m)}\right\}$ and $\mathscr{B} = \left\{E_1^{(\omega_1)}, \ldots , E_m^{(\omega_m)}\right\}$ are multisets over the edges in $E$, the multiset union of $\mathscr{A}$ and $\mathscr{B}$ is defined as $\mathscr{A} \sqcup \mathscr{B} = \left\{E_1^{(\mu_1+\omega_1)}, \ldots, E_m^{(\mu_m+\omega_m)}\right\}$. 

Let $\mathcal{H}_N = (V,E)$ be a network hypergraph. Consider the colors $c_1, \ldots, c_k$ and a multiset $\mathscr{E}$ over the edges in $E$. A \textit{k-coloring} of $\mathscr{E}$ is an assignment of the edges in $\mathscr{E}$ into $k$ (possibly empty) submultisets. We write $\mathscr{E} = \mathscr{E}_{c_1} \sqcup \cdots \sqcup \mathscr{E}_{c_k}$ and think of the edges in $\mathscr{E}_{c_i}$ as being colored $c_i$. The \textit{degree} of a vertex $v \in V$ with respect to $c_i$ is the number of edges in $\mathscr{E}_{c_i}$ that cover $v$ and will be denoted $\deg_{\mathscr{E}_{c_i}}(v)$. Every multiset $\mathscr{E}$ has many $k$-colorings, however, we are interested in 2-colorings that satisfy a particular property. When $\mathscr{E}$ has a 2-coloring, we take $c_1 = r$ to denote the color red and $c_2 = b$ to denote the color blue. 

\begin{defn}
Let $\mathcal{H} = (V,E)$ be a hypergraph. Suppose $\mathscr{E}$ is a multiset over the edges in $E$ and let $v \in V$.  We say that the vertex $v 
\in V$ is \textit{almost balanced with respect to the 2-coloring} $\mathscr{E} = \mathscr{E}_r \sqcup \mathscr{E}_b$ if: 
\begin{enumerate}
    \item $\deg_{\mathscr{E}_r}(v) = \deg_{\mathscr{E}_b}(v) + k$ for some $k \in \mathbb{Z}_{>0}$
    \item $\deg_{\mathscr{E}_r}(u) = \deg_{\mathscr{E}_b}(u)$ \; $\forall$ $u \in V \setminus \{v\}$.
\end{enumerate}
\end{defn}
In some cases, we omit the specification of the multiset $\mathscr{E}$ and simply say $v$ is almost balanced if there is some $\mathscr{E}$ such that $v$ is almost balanced with respect to a 2-coloring of $\mathscr{E}$.  If a vertex $v$ is almost balanced with respect to the multiset of edges $\mathscr{E}$, we say $\mathscr{E}$ is \textit{minimal} if it does not contain another multiset $\mathscr{E}'$ such that $v$ is almost balanced with respect to $\mathscr{E}'$. The definition of almost balanced is inspired by monomial hypergraphs and balanced edge sets as described in  \cite{PS} and \cite{gross2013combinatorial}.

\begin{prop}\label{almostbalanced} Let $N$ be a 0,1-network, $\mathscr{E}$ be a multiset over the edges in $E(\mathcal{H}_N)$, and $u_j \in V(\mathcal{H}_N)$. If $u_{j}$ is almost balanced with respect to the 2-coloring $\mathscr{E}=\mathscr{E}_r \sqcup \mathscr{E}_b$, then $k\kappa_{j}x^{y_{j}} = \sum_{E_s \in \mathcal{E}_b} \dot{x}_s - \sum_{E_s \in \mathcal{E}_r} \dot{x}_s$ for some positive integer $k$.
\end{prop}

\begin{proof}
Let $N = (\mathscr{S}, \mathscr{C}, \mathscr{R})$ be a 0,1-network and suppose $u_j$ is almost balanced with respect to the 2-coloring $\mathscr{E}=\mathscr{E}_r \sqcup \mathscr{E}_b$. By the definition of almost balanced, $\deg_{\mathscr{E}_r}(u_j) = \deg_{\mathscr{E}_b}(u_j) + k$ for some positive integer $k$ and $\deg_{\mathscr{E}_r}(w) = \deg_{\mathscr{E}_b}(w)$ for all $w \in V \setminus \{u_j\}$. Without loss of generality, suppose $\mathscr{E}$ is minimal. Then, for any edge $e \in \mathscr{E}$, either $e \in \mathscr{E}_r$ or $e \in \mathscr{E}_b$ but not both; otherwise, there is an edge $e' \in \mathscr{E}_r \sqcap \mathscr{E}_b$ and, consequently, $u_j$ is almost balanced with respect to the 2-coloring $\mathscr{E}' = (\mathscr{E}_r \setminus \{e'\}) \sqcup (\mathscr{E}_b \setminus \{e'\})$, contradicting the minimality of $\mathscr{E}$. In particular, each species edge $E_s \in \mathscr{E}$ is colored either red or blue but not both. Thus, the sets $\mathscr{S}_r = \{s \in \mathscr{S} : E_s \in \mathscr{E}_r\}$ and $\mathscr{S}_b = \{s \in \mathscr{S} : E_s \in \mathscr{E}_b\}$ are disjoint. We shall prove the following two claims:
\begin{enumerate}
    \item $\displaystyle\sum_{s \in \mathscr{S}_b} m_s\gamma_{s_j} - \sum_{s \in \mathscr{S}_r} m_s\gamma_{s_j} = k$
    \item $\displaystyle\sum_{s \in \mathscr{S}_b} m_s\gamma_{s_i} - \sum_{s \in \mathscr{S}_r} m_s\gamma_{s_i} = 0$ for all $i \neq j$,
\end{enumerate}
where $m_s$ denotes the multiplicity of the edge $E_s$ in $\mathscr{E}$. 

For $c \in \{r,b\}$, the sum $\sum_{s \in \mathscr{S}_c} m_s\gamma_{s_i}$ can be partitioned into four sums, one for each of the following situations for a given species $s$: 1) $u_{i} \in E_s$ and $v_{i} \not \in E_s$, 2) $u_{i} \not\in E_s$ and $v_{i} \in E_s$, 3) $u_{i},v_{i} \in E_s$, and 4) $u_{i},v_{i} \not\in E_s$. Then, for each $i$,
\begin{align*}
    \displaystyle\sum_{s \in \mathscr{S}_b} m_s\gamma_{s_i} - \sum_{s \in \mathscr{S}_r} m_s\gamma_{s_i}
    &= \sum_{\substack{s \in \mathscr{S}_b\\ u_{i} \in E_s\\ v_{i} \not \in E_s}} m_s\gamma_{s_i} + \sum_{\substack{s \in \mathscr{S}_b\\ u_{i} \not\in E_s\\ v_{i} \in E_s}} m_s\gamma_{s_i} + \sum_{\substack{s \in \mathscr{S}_b\\ u_{i} \in E_s\\ v_{i} \in E_s}} m_s\gamma_{s_i} + \sum_{\substack{s \in \mathscr{S}_b\\ u_{i} \not\in E_s\\ v_{i} \not\in E_s}} m_s\gamma_{s_i} \\
    &- \sum_{\substack{s \in \mathscr{S}_r\\ u_{i} \in E_s\\ v_{i} \not \in E_s}} m_s\gamma_{s_i} - \sum_{\substack{s \in \mathscr{S}_r\\ u_{i} \not\in E_s\\ v_{i} \in E_s}} m_s\gamma_{s_i} - \sum_{\substack{s \in \mathscr{S}_r\\ u_{i} \in E_s\\ v_{i} \in E_s}} m_s\gamma_{s_i} - \sum_{\substack{s \in \mathscr{S}_r\\ u_{i} \not\in E_s\\ v_{i} \not\in E_s}} m_s\gamma_{s_i}.
\end{align*}

Since $N$ is a 0,1-network, if $s$ is a species such that $u_{i}$ and $v_{i}$ both belong to $E_s$, then $\gamma_{s_i} = 0$. On the other hand, if $u_{i}$ and $v_{i}$ are both not in $E_s$, then $\gamma_{s_i} = 0$. Following suit, $u_{i} \in E_s$ and $v_{i} \not \in E_s$ implies $\gamma_{s_i} = -1$; $u_{i} \not\in E_s$ and $v_{i} \in E_s$ implies $\gamma_{s_i} = 1$. Then the previous equation becomes
\begin{align}
    \displaystyle\sum_{s \in \mathscr{S}_b} m_s\gamma_{s_i} - \sum_{s \in \mathscr{S}_r} m_s\gamma_{s_i} 
    &= - \sum_{\substack{s \in \mathscr{S}_b\\ u_{i} \in E_s\\ v_{i} \not \in E_s}} m_s + \sum_{\substack{s \in \mathscr{S}_b\\ u_{i} \not\in E_s\\ v_{i} \in E_s}} m_s + \sum_{\substack{s \in \mathscr{S}_r\\ u_{i} \in E_s\\ v_{i} \not \in E_s}} m_s - \sum_{\substack{s \in \mathscr{S}_r\\ u_{i} \not\in E_s\\ v_{i} \in E_s}} m_s . \label{eq:sumbycolor}
\end{align}

For each vertex $w$ in the reaction edge $E_i = \{u_{i},v_{i}\}$ where $w' \in E_i \setminus \{w\}$, if $m_{i}$ denotes the multiplicity of $E_{i}$ in $\mathscr{E}$, then notice that
\begin{equation}\label{red}
    \deg_{\mathscr{E}_r}(w) = m_i + \sum_{\substack{s \in \mathscr{S}_r\\ w \in E_s\\ w' \not\in E_s}} m_s + \sum_{\substack{s \in \mathscr{S}_r\\ w \in E_s\\ w' \in E_s}} m_s
\end{equation}
and 
\begin{equation}\label{blue}
    \deg_{\mathscr{E}_b}(w) = m_i + \sum_{\substack{s \in \mathscr{S}_b\\ w \in E_s\\ w' \not\in E_s}} m_s + \sum_{\substack{s \in \mathscr{S}_b\\ w \in E_s\\ w' \in E_s}} m_s .
\end{equation}

If $i \neq j$, then since $\deg_{\mathscr{E}_r}(u_{i}) = \deg_{\mathscr{E}_b}(u_{i})$ and $\deg_{\mathscr{E}_r}(v_{i}) = \deg_{\mathscr{E}_b}(v_{i})$, the following equation holds
\begin{equation}\label{difference}
     0 = - \deg_{\mathscr{E}_b}(u_{i}) + \deg_{\mathscr{E}_b}(v_{i}) + \deg_{\mathscr{E}_r}(u_{i}) - \deg_{\mathscr{E}_r}(v_{i}).
\end{equation}
By substituting the equations \eqref{red} and \eqref{blue} into \eqref{difference}, we have 
\begin{align}
    0 &= - \left( m_i + \sum_{\substack{s \in \mathscr{S}_b\\ u_i \in E_s\\ v_i \not\in E_s}} m_s + \sum_{\substack{s \in \mathscr{S}_b\\ u_i \in E_s\\ v_i \in E_s}} m_s \right) + m_i + \sum_{\substack{s \in \mathscr{S}_b\\ u_i \not\in E_s\\ v_i \in E_s}} m_s + \sum_{\substack{s \in \mathscr{S}_b\\ u_i \in E_s\\ v_i \in E_s}} m_s\\
    &+ m_i + \sum_{\substack{s \in \mathscr{S}_r\\ u_i \in E_s\\ v_i \not\in E_s}} m_s + \sum_{\substack{s \in \mathscr{S}_r\\ u_i \in E_s\\ v_i \in E_s}} m_s - \left( m_i + \sum_{\substack{s \in \mathscr{S}_r\\ u_i \not\in E_s\\ v_i \in E_s}} m_s + \sum_{\substack{s \in \mathscr{S}_r\\ u_i \in E_s\\ v_i \in E_s}} m_s \right) \\
    &=  -\sum_{\substack{s \in \mathscr{S}_b\\ u_i \in E_s\\ v_i \not\in E_s}} m_s + \sum_{\substack{s \in \mathscr{S}_b\\ u_i \not\in E_s\\ v_i \in E_s}} m_s + \sum_{\substack{s \in \mathscr{S}_r\\ u_i \in E_s\\ v_i \not\in E_s}} m_s - \sum_{\substack{s \in \mathscr{S}_r\\ u_i \not\in E_s\\ v_i \in E_s}} m_s. \label{eq:sumequaltozero}
\end{align}
On the other hand, if $i = j$, then in a similar way the equations $\deg_{\mathscr{E}_r}(u_j) = \deg_{\mathscr{E}_b}(u_j) + k$ and $\deg_{\mathscr{E}_r}(v_{j}) = \deg_{\mathscr{E}_b}(v_{j})$ together with \eqref{red} and \eqref{blue} imply 
\begin{equation}
 k = -\sum_{\substack{s \in \mathscr{S}_b\\ u_j \in E_s\\ v_{j} \not\in E_s}} m_s + \sum_{\substack{s \in \mathscr{S}_b\\ u_j \not\in E_s\\ v_{j} \in E_s}} m_s + \sum_{\substack{s \in \mathscr{S}_r\\ u_j \in E_s\\ v_{j} \not\in E_s}} m_s - \sum_{\substack{s \in \mathscr{S}_r\\ u_j \not\in E_s\\ v_{j} \in E_s}} m_s. \label{eq:sumtok}
\end{equation}
Now using \eqref{eq:sumbycolor}, \eqref{eq:sumequaltozero}, and \eqref{eq:sumtok}, we obtain
\begin{align*}
    k\kappa_jx^{y_j} 
    &= \kappa_jx^{y_j}\left( \displaystyle\sum_{s \in \mathscr{S}_b} m_s\gamma_{s_j} - \sum_{s \in \mathscr{S}_r} m_s\gamma_{s_j} \right) \\
    &+ \sum_{\substack{y_i \to y_i'\\ i \neq j}} \kappa_{i}x^{y_i} \left(\sum_{s \in \mathscr{S}_b} m_s\gamma_{s_i} - \sum_{s \in \mathscr{S}_r} m_s\gamma_{s_i} \right) \\
    &= \sum_{y_i \to y_i'} \kappa_{i}x^{y_i} \left(\sum_{s \in \mathscr{S}_b} m_s\gamma_{s_i} - \sum_{s \in \mathscr{S}_r} m_s\gamma_{s_i} \right)\\
    &= \sum_{y_i \to y_i'} \kappa_{i}x^{y_i} \sum_{s \in \mathscr{S}_b} m_s\gamma_{s_i} - \sum_{y_i \to y_i'} \kappa_{i}x^{y_i} \sum_{s \in \mathscr{S}_r} m_s\gamma_{s_i}\\
    &= \sum_{s \in \mathscr{S}_b} m_s \sum_{y_i \to y_i'} \gamma_{s_i}\kappa_{i}x^{y_i} - \sum_{s \in \mathscr{S}_r} m_s \sum_{y_i \to y_i'} \gamma_{s_i}\kappa_{i}x^{y_i}\\
    &= \sum_{s \in \mathscr{S}_b} m_s \dot{x}_s - \sum_{s \in \mathscr{S}_r} m_s \dot{x}_s\\
    &= \sum_{E_s \in \mathscr{E}_b} \dot{x}_s - \sum_{E_s \in \mathscr{E}_r} \dot{x}_s. 
\end{align*}
\end{proof}

\begin{example} Consider the $0,1$-network pictured in Figure \ref{examplehypergraph} its network hypergraph.

\begin{figure}[H]
    \begin{multicols}{2}
    \noindent
    \centering
    
    \begin{tikzcd}
        A+B \ar[r,"\kappa_1"] & D\\
        A+C \arrow[r, "\kappa_2"] & D\\
        B+C \ar[r,"\kappa_3"] & D
    \end{tikzcd}
    
    \begin{tikzpicture}[scale = .6]
    \node (AB) at (0,4) {};
    \node (AC) at (-.5,2) {};
    \node (BC) at (.5,0) {};
    \node (N1) at (2.5,4) {};
    \node (N2) at (3,2) {};
    \node (N3) at (2.7,0) {};
    
    \begin{scope}[fill opacity=0.5]
    \filldraw[fill=violet!70] ($(AB)+(0,.2)$)
        to[out=0,in=180] ($(N1)+(0,.2)$)
        to[out=0,in=90] ($(N1)+(.2,0)$)
        to[out=270,in=0] ($(N1)+(0,-.2)$)
        to[out=180,in=0] ($(AB)+(0,-.2)$)
        to[out=180,in=270] ($(AB)+(-.2,0)$)
        to[out=90,in=180] ($(AB)+(0,.2)$);
    \filldraw[fill=brown!90] ($(AC)+(0,.2)$)
        to[out=0,in=180] ($(N2)+(0,.2)$)
        to[out=0,in=90] ($(N2)+(.2,0)$)
        to[out=270,in=0] ($(N2)+(0,-.2)$)
        to[out=180,in=0] ($(AC)+(0,-.2)$)
        to[out=180,in=270] ($(AC)+(-.2,0)$)
        to[out=90,in=180] ($(AC)+(0,.2)$);
    \filldraw[fill=magenta!80] ($(BC)+(0,.2)$)
        to[out=0,in=180] ($(N3)+(0,.2)$)
        to[out=0,in=90] ($(N3)+(.2,0)$)
        to[out=270,in=0] ($(N3)+(0,-.2)$)
        to[out=180,in=0] ($(BC)+(0,-.2)$)
        to[out=180,in=270] ($(BC)+(-.2,0)$)
        to[out=90,in=180] ($(BC)+(0,.2)$);
    \filldraw[fill=yellow] ($(AB) + (0,0.45)$)
        to[out=20,in=90] ($(1,3)+(.5,0)$)
        to[out=270,in=90] (1.25,1.25)
        to[out=270,in=0] ($(BC) + (0,-0.5)$)
        to[out=180,in=270] ($(BC) + (-.5,0)$)
        to[out=90,in=270] ($(1.25,2)+(-.5,0)$)
        to[out=90,in=340] ($(AB) + (0,-0.35)$)
        to[out=160,in=270] ($(AB) + (-0.3,0)$)
        to[out=90,in=200] ($(AB) + (0,0.45)$);
    \filldraw[fill=orange!90] ($(AC) + (0,.3)$)
        to[out=0,in=135] ($(BC) + (.1,.3)$)
        to[out=315,in=90] ($(BC) + (.3,0)$)
        to[out=270,in=0] ($(BC) + (0,-.3)$)
        to[out=180,in=280] ($(AC) + (-.3,-.3)$)
        to[out=100,in=270] ($(AC) + (-.3,0)$)
        to[out=90,in=180] ($(AC) + (0,.3)$);
    \filldraw[fill=pink] ($(AB) + (0,0.5)$)
        to[out=0,in=90] ($(AB) + (.5,0)$)
        to[out=270,in=90] ($(-.25,3) + (.35,0)$)
        to[out=270,in=0] ($(AC) + (0,-0.5)$)
        to[out=180,in=270] ($(AC) + (-0.5,0)$)
        to[out=90,in=260] ($(-.25,3) + (-.35,0)$)
        to[out=80,in=270] ($(AB) + (-0.5,0)$)
        to[out=90,in=180] ($(AB) + (0,0.5)$);
    \filldraw[fill=green!60] ($(N1)+(.05,.4)$)
        to[out=0,in=90] ($(N1)+(.4,0)$)
        to[out=270,in=90] ($(N2)+(.4,0)$)
        to[out=270,in=90] ($(N3)+(.4,0)$)
        to[out=270,in=0] ($(N3)+(0,-.3)$)
        to[out=180,in=270] ($(N3)+(-.3,0)$)
        to[out=90,in=270] ($(N2)+(-.3,0)$)
        to[out=90,in=270] ($(N1)+(-.3,0)$)
        to[out=90,in=180] ($(N1)+(.05,.4)$);
    \end{scope}
    
    \fill (AB) circle (0.1);
    \fill (AC) circle (0.1);
    \fill (BC) circle (0.1);
    
    \fill (N1) circle (0.1);
    \fill (N2) circle (0.1);
    \fill (N3) circle (0.1);
    
    \fill ($(AB)+(-1,0)$) node {$u_1$};
    \fill ($(AC)+(-1,0)$) node {$u_2$};
    \fill ($(BC)+(-1.5,0)$) node {$u_3$};
    \fill ($(N1) + (.75,0)$) node {$v_1$};
    \fill ($(N2)+ (.75,0)$) node {$v_2$};
    \fill ($(N3) + (.75,0)$) node {$v_3$};
    \fill (-.25,3) node {$\scaleobj{.8}{E_{A}}$};
    \fill (1.1,3) node {$\scaleobj{.8}{E_{B}}$};
    \fill (-.15,1) node {$\scaleobj{.8}{E_{C}}$};
    \fill (2.9,1) node {$\scaleobj{.8}{E_{D}}$};
    \fill (N1) node [xshift=-.5cm,yshift=.3cm] {$\scaleobj{.8}{E_1}$};
    \fill (N2) node [xshift=-.5cm,yshift=.3cm] {$\scaleobj{.8}{E_2}$};
    \fill (N3) node [xshift=-.5cm,yshift=.3cm] {$\scaleobj{.8}{E_3}$};
    \end{tikzpicture}
    \end{multicols}
    
    \caption{A chemical reaction network and its network hypergraph.}
    \label{examplehypergraph}
\end{figure}
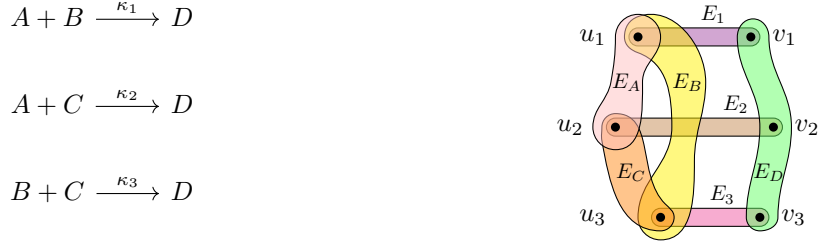

The vertex $u_1$ is almost balanced with respect to the 2-coloring of $\mathscr{E} = \{E_A,E_B,E_C\}$ defined by $\mathscr{E}_r = \{E_A,E_B\}$ and $\mathscr{E}_b = \{E_C\}$; see Figure \ref{fig:bicoloring1}. Further, observe that:
\begin{align*}
    \sum_{E_s \in \mathscr{E}_b} \dot{x}_s - \sum_{E_s \in \mathscr{E}_r} \dot{x}_s 
    &= \dot{x}_C - \dot{x}_A - \dot{x}_B\\
    &= (-\kappa_2x_Ax_C - \kappa_3x_Bx_C) \\
    &\hspace{.5cm} - (-\kappa_1x_Ax_B - \kappa_2x_Ax_C)\\
    &\hspace{.5cm} - (-\kappa_1x_Ax_B - \kappa_3x_Bx_C)\\
    &= 2\kappa_1x_Ax_B.
\end{align*}

\begin{figure}[H]
    \centering
    \begin{tikzpicture}[scale = .6]
    \node (AB) at (0,4) {};
    \node (AC) at (-.5,2) {};
    \node (BC) at (.5,0) {};
    \node (N1) at (2.5,4) {};
    \node (N2) at (3,2) {};
    \node (N3) at (2.7,0) {};
    
    \begin{scope}[fill opacity=0.5]
    \filldraw[fill=red!75] ($(AB) + (0,0.45)$)
        to[out=20,in=90] ($(1,3)+(.5,0)$)
        to[out=270,in=90] (1.25,1.25)
        to[out=270,in=0] ($(BC) + (0,-0.5)$)
        to[out=180,in=270] ($(BC) + (-.5,0)$)
        to[out=90,in=270] ($(1.25,2)+(-.5,0)$)
        to[out=90,in=340] ($(AB) + (0,-0.35)$)
        to[out=160,in=270] ($(AB) + (-0.3,0)$)
        to[out=90,in=200] ($(AB) + (0,0.45)$);
    \filldraw[fill=blue!50] ($(AC) + (0,.3)$)
        to[out=0,in=135] ($(BC) + (.1,.3)$)
        to[out=315,in=90] ($(BC) + (.3,0)$)
        to[out=270,in=0] ($(BC) + (0,-.3)$)
        to[out=180,in=280] ($(AC) + (-.3,-.3)$)
        to[out=100,in=270] ($(AC) + (-.3,0)$)
        to[out=90,in=180] ($(AC) + (0,.3)$);
    \filldraw[fill=red!75] ($(AB) + (0,0.5)$)
        to[out=0,in=90] ($(AB) + (.5,0)$)
        to[out=270,in=90] ($(-.25,3) + (.35,0)$)
        to[out=270,in=0] ($(AC) + (0,-0.5)$)
        to[out=180,in=270] ($(AC) + (-0.5,0)$)
        to[out=90,in=260] ($(-.25,3) + (-.35,0)$)
        to[out=80,in=270] ($(AB) + (-0.5,0)$)
        to[out=90,in=180] ($(AB) + (0,0.5)$);
    \end{scope}
    \fill (AB) circle (0.1);
    \fill (AC) circle (0.1);
    \fill (BC) circle (0.1);
    \fill (N1) circle (0.1);
    \fill (N2) circle (0.1);
    \fill (N3) circle (0.1);
    \fill ($(AB)+(-1,0)$) node {$u_1$};
    \fill ($(AC)+(-1,0)$) node {$u_2$};
    \fill ($(BC)+(-1.5,0)$) node {$u_3$};
    \fill ($(N1) + (.75,0)$) node {$v_1$};
    \fill ($(N2)+ (.75,0)$) node {$v_2$};
    \fill ($(N3) + (.75,0)$) node {$v_3$};
    \fill (-.25,3) node {$\scaleobj{.8}{E_{A}}$};
    \fill (1.1,3) node {$\scaleobj{.8}{E_{B}}$};
    \fill (-.15,1) node {$\scaleobj{.8}{E_{C}}$};
    \end{tikzpicture}
    \caption{A 2-coloring of the edges $\{E_A,E_B,E_C\}$ from the hypergraph pictured in Figure \ref{examplehypergraph}.}
    \label{fig:bicoloring1}
\end{figure}
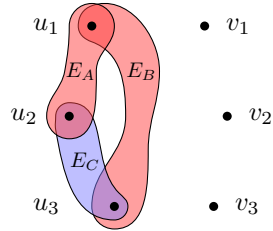
\end{example}

The previous proposition concerned vertices in $\mathcal H_N$ corresponding to reactant complexes, using similar reasoning, we get a comparable proposition regarding vertices in $\mathcal H_N$ corresponding to product complexes.
\begin{prop}\label{prodalmostbalanced} Let $N$ be a 0,1-network, $\mathscr{E}$ be a multiset over the edges in $E(\mathcal{H}_N)$, and $v_j \in V(\mathcal{H}_N)$. If $v_{j}$ is almost balanced with respect to the 2-coloring $\mathscr{E}=\mathscr{E}_r \sqcup \mathscr{E}_b$, then $k\kappa_{j}x^{y_{j}} = \sum_{E_s \in \mathcal{E}_r} \dot{x}_s - \sum_{E_s \in \mathcal{E}_b} \dot{x}_s$ for some positive integer $k$.
\end{prop}

\begin{corollary} \label{monomial_in_ideal}
Let $N$ be a 0,1-network. If either vertex $u_{i}$ or $v_{i}$ in $V(\mathcal{H}_N)$ is almost balanced, then $x^{y_i} \in \mathcal{I}(N)$.
\end{corollary}

\begin{prop}\label{swap_colors}
Let $N$ be a chemical reaction network containing the reaction $y \to y'$ such that $y' \neq \varnothing$. Then the vertex $u \in V(\mathcal{H})$ corresponding to $y$ in the reaction $y \to y'$ is almost balanced if and only if the vertex $v \in V(\mathcal{H})$ corresponding to $y'$ in the reaction $y \to y'$ is almost balanced.
\end{prop}

\begin{proof}
Let $N = (\mathscr{S}, \mathscr{C}, \mathscr{R})$ be a chemical reaction network such that $y \to y' \in \mathscr{R}$ and $y' \neq \varnothing$. Let $u$ and $v$ be distinct vertices in $V(\mathcal{H})$ corresponding to the complexes $y$ and $y'$ of the reaction $y \to y'$, respectively. Since $y' \neq \varnothing$, then the hyperedge $e \in E(\mathcal{H})$ corresponding to the reaction $y \to y'$ is nonempty, that is, $e = \{u,v\}$; see Figure \ref{fig:swapcolors}. Assume $w \in e$ is almost balanced. We show that $w' \in e \setminus \{w\}$ is almost balanced.

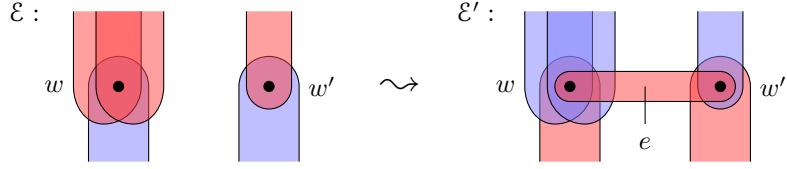
\begin{figure}[!htb]
    \captionsetup{justification=centering,margin=2cm}
    \begin{center}
    \begin{tikzpicture}
\node (u1) at (0,0) {};
\node (v1) at (2,0) {};
\node (u2) at (6,0) {};
\node (v2) at (8,0) {};

\begin{scope}[fill opacity = .5]
\filldraw[fill=blue!50] 
    ($(u1)+(-.4,-1)$) 
    to[out=90,in=270] ($(u1)+(-.4,0)$)
    to[out=90,in=180] ($(u1)+(0,.4)$)
    to[out=0,in=90] ($(u1)+(.4,0)$)
    to[out=270,in=90] ($(u1)+(.4,-1)$);
\filldraw[fill=red!75] 
    ($(u1)+(-.6,1)$) 
    to[out=270,in=90] ($(u1)+(-.6,0)$)
    to[out=270,in=180] ($(u1)+(-.2,-.5)$)
    to[out=0,in=270] ($(u1)+(.3,0)$)
    to[out=90,in=270] ($(u1)+(.3,1)$);
\filldraw[fill=red!75] 
    ($(u1)+(-.3,1)$) 
    to[out=270,in=90] ($(u1)+(-.3,0)$)
    to[out=270,in=180] ($(u1)+(.2,-.5)$)
    to[out=0,in=270] ($(u1)+(.6,0)$)
    to[out=90,in=270] ($(u1)+(.6,1)$);

\filldraw[fill=red!75] 
    ($(u2)+(-.4,-1)$) 
    to[out=90,in=270] ($(u2)+(-.4,0)$)
    to[out=90,in=180] ($(u2)+(0,.4)$)
    to[out=0,in=90] ($(u2)+(.4,0)$)
    to[out=270,in=90] ($(u2)+(.4,-1)$);
\filldraw[fill=blue!50] 
    ($(u2)+(-.6,1)$) 
    to[out=270,in=90] ($(u2)+(-.6,0)$)
    to[out=270,in=180] ($(u2)+(-.2,-.5)$)
    to[out=0,in=270] ($(u2)+(.3,0)$)
    to[out=90,in=270] ($(u2)+(.3,1)$);
\filldraw[fill=blue!50] 
    ($(u2)+(-.3,1)$) 
    to[out=270,in=90] ($(u2)+(-.3,0)$)
    to[out=270,in=180] ($(u2)+(.2,-.5)$)
    to[out=0,in=270] ($(u2)+(.6,0)$)
    to[out=90,in=270] ($(u2)+(.6,1)$);
    
\filldraw[fill=blue!50] 
    ($(v1)+(-.4,-1)$) 
    to[out=90,in=270] ($(v1)+(-.4,0)$)
    to[out=90,in=180] ($(v1)+(0,.4)$)
    to[out=0,in=90] ($(v1)+(.4,0)$)
    to[out=270,in=90] ($(v1)+(.4,-1)$);
\filldraw[fill=red!75] 
    ($(v1)+(-.3,1)$) 
    to[out=270,in=90] ($(v1)+(-.3,0)$)
    to[out=270,in=180] ($(v1)+(0,-.3)$)
    to[out=0,in=270] ($(v1)+(.3,0)$)
    to[out=90,in=270] ($(v1)+(.3,1)$);
    
\filldraw[fill=red!75] 
    ($(v2)+(-.4,-1)$) 
    to[out=90,in=270] ($(v2)+(-.4,0)$)
    to[out=90,in=180] ($(v2)+(0,.4)$)
    to[out=0,in=90] ($(v2)+(.4,0)$)
    to[out=270,in=90] ($(v2)+(.4,-1)$);
\filldraw[fill=blue!50] 
    ($(v2)+(-.3,1)$) 
    to[out=270,in=90] ($(v2)+(-.3,0)$)
    to[out=270,in=180] ($(v2)+(0,-.3)$)
    to[out=0,in=270] ($(v2)+(.3,0)$)
    to[out=90,in=270] ($(v2)+(.3,1)$);
    
\filldraw[fill=red!75] 
        ($(u2) + (0,.2)$)
        to[out=0,in=180] ($(v2) + (0,.2)$)
        to[out=0,in=90] ($(v2) + (.2,0)$)
        to[out=270,in=0] ($(v2) + (0,-.2)$)
        to[out=180,in=0] ($(u2) + (0,-.2)$)
        to[out=180,in=270] ($(u2) + (-.2,0)$)
        to[out=90,in=180] ($(u2) + (0,.2)$);
\end{scope}

\fill (u1) circle (.075);
\fill (v1) circle (.075);
\fill (u2) circle (.075);
\fill (v2) circle (.075);
\draw ($(u2)+(1,0)$) -- ($(u2)+(1,-.5)$);

\fill (u1) node [left, xshift=-.6cm] {$w$};
\fill (v1) node [right, xshift=.4cm] {$w'$};
\fill (u2) node [left, xshift=-.6cm] {$w$};
\fill (v2) node [right, xshift=.4cm] {$w'$};
\fill ($(u2)+(1,-.75)$) node {$e$};
\fill (3.75,0) node {$\scaleobj{1.5}{\leadsto}$};
\fill ($(u1)+(-1.25,1)$) node {$\mathscr{E}:$};
\fill ($(u2)+(-1.25,1)$) node {$\mathscr{E}':$};
\end{tikzpicture}
    \end{center}
    \caption{Obtaining a 2-coloring of $\mathscr{E}'$ from a 2-coloring of $\mathscr{E}$ by swapping colors and adding red reactions edges.}
    \label{fig:swapcolors}
\end{figure}

Suppose $w$ is almost balanced with respect to the 2-colored multiset $\mathscr{E} = \mathscr{E}_r \sqcup \mathscr{E}_b$ over the edges in $E(\mathcal{H})$. Hence, $\deg_{\mathscr{E}_r}(w) = \deg_{\mathscr{E}_b}(w)+k$ for some positive integer $k$ and $\deg_{\mathscr{E}_r}(z) = \deg_{\mathscr{E}_b}(z)$ for all $z \in V(\mathcal{H}) \setminus \{w\}$. Set $\mathscr{E}' = \mathscr{E} \sqcup \left\{e^{(k)}\right\}$. 2-color the edges in $\mathscr{E}'$ so that $\mathscr{E}_r' = \mathscr{E}_b \sqcup \left\{ e^{(k)} \right\}$ and $\mathscr{E}_b' = \mathscr{E}_r$. Since $e$ covers only $w$ and $w'$, then $\deg_{\mathscr{E}_r'}(w) = \deg_{\mathscr{E}_b}(w) + k = \deg_{\mathscr{E}_r}(w) = \deg_{\mathscr{E}_b'}(w)$, $\deg_{\mathscr{E}_r'}(w') = \deg_{\mathscr{E}_b}(w') + k = \deg_{\mathscr{E}_r}(w') + k = \deg_{\mathscr{E}_b'}(w')+k$, and $\deg_{\mathscr{E}_r'}(z) = \deg_{\mathscr{E}_b}(z) = \deg_{\mathscr{E}_r}(z) = \deg_{\mathscr{E}_b'}(z)$ for all $z \in V(\mathcal{H})\setminus \{w,w'\}$. Therefore, $w'$ is almost balanced with respect to $\mathscr{E}'$.
\end{proof}

A reaction $y_i \to y_i' \in \mathscr{R}$ is said to be \textit{reversible} if $y_i' \to y_i \in \mathscr{R}$. A network $N = (\mathscr{S}, \mathscr{C}, \mathscr{R})$ is \textit{reversible} if every reaction $y \to y' \in \mathscr{R}$ is reversible; $N$ is said to be \textit{weakly reversible} if the reaction graph is strongly connected, that is, if for any two complexes $y$ and $y'$ in $\mathscr{C}$, the existence of a sequence $y_1, \ldots, y_j \in \mathscr{C}$ such that $y=y_1 \to \cdots \to y_j=y'$ implies the existence of a sequence $y_1', \ldots, y_k' \in \mathscr{C}$ such that $y'=y_1' \to \cdots \to y_k'=y$. 

In a reversible reaction we require both complexes belonging to a reversible reaction to be nonempty since we are not considering reactions of the form $\varnothing \to y$. Vertices in the network hypergraph that come from complexes belonging to a reversible reaction are never almost balanced.

\begin{prop}\label{reversible_is_balanced}
Let $N$ be a chemical reaction network. If $y \to y' \in \mathscr{R}$ is reversible, then the vertices in $V(\mathcal{H}_N)$ corresponding to either $y$ or $y'$ in the reaction $y \to y'$ or $y' \to y$ is not almost balanced.
\end{prop}

\begin{proof} Let $N = (\mathscr{S}, \mathscr{C}, \mathscr{R})$ be a chemical reaction network such that $\mathscr{R}$ contains the reactions $y \to y'$ and $y' \to y$. Let $u,v,u',v'$ be vertices in $V(\mathcal{H}_N)$ such that $u$ and $v$ correspond to $y$ and $y'$ in the reaction $y \to y'$, respectively, and $u'$ and $v'$ correspond to $y'$ and $y$ in the reaction $y' \to y$, respectively. Note that $u$ and $v'$ correspond to the same complex $y$ but in different reactions. Thus, for any $s \in \supp{y}$, $E_s$ covers $u$ if and only if $E_s$ covers $v'$; see the yellow edge in Figure \ref{fig:reversible_subhypergraph}. Similarly, both $u'$ and $v$ correspond to  $y'$ so, for any $s \in \supp{y'}$, $E_s$ covers $u'$ if and only if $E_s$ covers $v$; see the orange edge in Figure \ref{fig:reversible_subhypergraph}. 

By the symmetry of the reversible reactions, it is enough to show that $u$ is not almost balanced. Seeking a contradiction, suppose $u$ is almost balanced with respect to the 2-colored multiset $\mathscr{E} = \mathscr{E}_r \sqcup \mathscr{E}_b$ over the edges in $E(\mathcal{H}_N)$. Then $\deg_{\mathscr{E}_r}(u) = \deg_{\mathscr{E}_b}(u) + k$ for some positive integer $k$ and $\deg_{\mathscr{E}_r}(w) = \deg_{\mathscr{E}_b}(w)$ for all $w \in V(\mathcal{H}_N) \setminus \{u\}$.

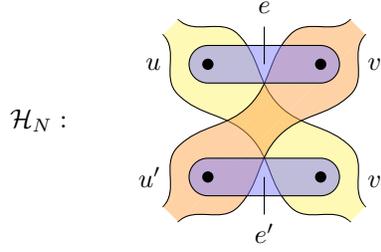
\begin{figure}[!htb]
    \captionsetup{justification=centering,margin=2cm}
    \begin{center}
    \begin{tikzpicture}
\node (u1) at (0,1.5) {};
\node (u2) at (0,0) {};
\node (v1) at (1.5,1.5) {};
\node (v2) at (1.5,0) {};
\node (l) at (-1,.75) {};
\node (r) at (2.5,.75) {};
\node (m) at (.75,.75) {};

\begin{scope}[fill opacity=.5]

\filldraw[fill=yellow!70]
        ($(u1)+(-.4,.6)$)
        to[out=-45,in=135] ($(u1)+(.4,.4)$)
        to[out=-45,in=135] ($(m)+(.2,.2)$)
        to[out=-45,in=135] ($(v2)+(.4,.4)$)
        to[out=-45,in=135] ($(v2)+(.6,-.4)$);
\filldraw[fill=yellow!70]
        ($(u1)+(-.6,.4)$)
        to[out=-45,in=135] ($(u1)+(-.4,-.4)$)
        to[out=-45,in=135] ($(m)+(-.2,-.2)$)
        to[out=-45,in=135] ($(v2)+(-.4,-.4)$)
        to[out=-45,in=135] ($(v2)+(.4,-.6)$);
\fill[yellow!70]
    let \p1=($(u1)!-7mm!(v2)$),
        \p2=($(v2)!-7mm!(u1)$),
        \p3=($(\p1)!1.413mm!90:(\p2)$),
        \p4=($(\p1)!1.413mm!-90:(\p2)$),
        \p5=($(\p2)!1.413mm!90:(\p1)$),
        \p6=($(\p2)!1.413mm!-90:(\p1)$)
    in
    (\p3) -- (\p4)-- (\p5) -- (\p6) -- cycle;
    
\filldraw[fill=orange!70]
        ($(v1)+(.4,.6)$)
        to[out=-135,in=45] ($(v1)+(-.4,.4)$)
        to[out=-135,in=45] ($(m)+(-.2,.2)$)
        to[out=-135,in=45] ($(u2)+(-.4,.4)$)
        to[out=-135,in=45] ($(u2)+(-.6,-.4)$);
\filldraw[fill=orange!70]
        ($(v1)+(.6,.4)$)
        to[out=-135,in=45] ($(v1)+(.4,-.4)$)
        to[out=-135,in=45] ($(m)+(.2,-.2)$)
        to[out=-135,in=45] ($(u2)+(.4,-.4)$)
        to[out=-135,in=45] ($(u2)+(-.4,-.6)$);
\fill[orange!70]
    let \p1=($(v1)!-7mm!(u2)$),
        \p2=($(u2)!-7mm!(v1)$),
        \p3=($(\p1)!1.413mm!90:(\p2)$),
        \p4=($(\p1)!1.413mm!-90:(\p2)$),
        \p5=($(\p2)!1.413mm!90:(\p1)$),
        \p6=($(\p2)!1.413mm!-90:(\p1)$)
    in
    (\p3) -- (\p4)-- (\p5) -- (\p6) -- cycle;

\filldraw[fill=blue!50] 
        ($(u1) + (0,0.25)$)
        to[out=0,in=180] ($(v1) + (0,0.25)$)
        to[out=0,in=90] ($(v1) + (0.25,0)$)
        to[out=270,in=0] ($(v1) + (0,-0.25)$)
        to[out=180,in=0] ($(u1) + (0,-0.25)$)
        to[out=180,in=270] ($(u1) + (-0.25,0)$)
        to[out=90,in=180] ($(u1) + (0,0.25)$);
\filldraw[fill=blue!50] 
        ($(u2) + (0,0.25)$)
        to[out=0,in=180] ($(v2) + (0,0.25)$)
        to[out=0,in=90] ($(v2) + (0.25,0)$)
        to[out=270,in=0] ($(v2) + (0,-0.25)$)
        to[out=180,in=0] ($(u2) + (0,-0.25)$)
        to[out=180,in=270] ($(u2) + (-0.25,0)$)
        to[out=90,in=180] ($(u2) + (0,0.25)$);
\end{scope}

\fill (u1) circle (.075);
\fill (u2) circle (.075);
\fill (v1) circle (.075);
\fill (v2) circle (.075);

\draw ($(m)+(0,.75)$) -- ($(m)+(0,1.25)$);
\draw ($(m)+(0,-.75)$) -- ($(m)+(0,-1.25)$);

\fill (u1) node [left,xshift=-.5cm] {$u$};
\fill (u2) node [left,xshift=-.5cm] {$u'$};
\fill (v1) node [right,xshift=.5cm] {$v$};
\fill (v2) node [right,xshift=.5cm] {$v'$};
\fill ($(m)+(-3,0)$) node {$\mathcal{H}_N:$};
\fill ($(m)+(0,-1.5)$) node {$e'$};
\fill ($(m)+(0,1.5)$) node {$e$};
\end{tikzpicture}
    \end{center}
    \caption{A subhypergraph of a network hypergraph corresponding to a network containing a reversible reaction.}
    \label{fig:reversible_subhypergraph}
\end{figure}

For each $w \in V(\mathcal{H})$ and $c \in \{r,b\}$, let $\sigma_c(w)$ be the number of species edges in $\mathscr{E}_c$ that cover $w$ and let $\rho_c(w)$ be the number of reaction edges in $\mathscr{E}_c$ that cover $w$. By a previous observation, for any $c \in \{r,b\}$, $\sigma_c(u) = \sigma_c(v')$ and $\sigma_c(u') = \sigma_c(v)$. By definition of the reaction edge $e \defeq \{u,v\}$, $e$ covers $u$ if and only if $e$ covers $v$. Similarly, the reaction edge $e' \defeq \{u',v'\}$ covers $u'$ if and only if $e'$ covers $v'$. Thus, for any $c \in \{r,b\}$, $\rho_c(u)=\rho_c(v)$ and $\rho_c(u')=\rho_c(v')$. Since $\deg_{\mathscr{E}_r}(w) = \deg_{\mathscr{E}_b}(w)$, for any $w \in \{v,u'v'\}$, then $\sigma_r(w) + \rho_r(w) = \sigma_b(w) + \rho_b(w)$ or, equivalently, $\sigma_r(w) - \sigma_b(w) = \rho_b(w) - \rho_r(w)$. Therefore,
\begin{align*}
    k &= \deg_{\mathscr{E}_r}(u) - \deg_{\mathscr{E}_b}(u)\\
    &= \sigma_r(u) + \rho_r(u) - \sigma_b(u) - \rho_b(u)\\
    &= \sigma_r(v') + \rho_r(v) - \sigma_b(v') - \rho_b(v)\\
    &= \rho_b(v') + \sigma_b(v) - \rho_r(v') - \sigma_r(v)\\
    &= \rho_b(u') + \sigma_b(u') - \rho_r(u') - \sigma_r(u')\\
    &= \deg_{\mathscr{E}_b}(u') - \deg_{\mathscr{E}_r}(u')\\
    &= 0.
\end{align*}
This is a contradiction since $k$ is a strictly positive integer so the assumption that $u$ is almost balanced must be false.
\end{proof}

A generalization of a reversible reaction is the notion of a reaction cycle. A \textit{reaction cycle} is a sequence of reactions $y_1 \to \cdots \to y_\ell \to y_1$ such that $\ell \geq 2$ and the complexes $y_1, \ldots, y_\ell$ are distinct. The length of a cycle is the number of reactions belonging to the cycle. Note that if $C$ is a reaction cycle of length 2, then $C$ consists of two reversible reactions. Proposition \ref{reversible_is_balanced} can be generalized to reaction cycles of any finite length.

\begin{prop}
Let $C$ be a reaction cycle of length $\ell \geq 2$ in the chemical reaction network $N$. Then the vertices corresponding to complexes in the reactions on $C$ are not almost balanced.
\end{prop}

\begin{proof}
Let $C = y_1 \to \cdots \to y_\ell \to y_1$ be a reaction cycle of length $\ell \geq 2$ in the chemical reaction network $N = (\mathscr{S}, \mathscr{C}, \mathscr{R})$. We will show the vertices $u_1,v_1, \ldots, u_\ell,v_\ell \in V(\mathcal{H}_N)$ are not almost balanced where $u_i$ and $v_i$ corrrespond to the reactant and product complexes of the reaction $y_i \to y_{i+1}$, respectively, and the indices are taken modulo $\ell$. By the symmetry of the subhypergraph of $\mathcal{H}_N$ induced by $C$, it is enough to show $u_1$ is not almost balanced; see Figure \ref{fig:reaction_cycle}. Seeking a contradiction, suppose $u_1$ is almost balanced. Then there is an edgeset $\mathscr{E}$ such that $u_1$ is almost balanced with respect to the 2-coloring $\mathscr{E} = \mathscr{E}_r \sqcup \mathscr{E}_b$.

\begin{figure}[h]
    \centering
    \begin{tikzpicture}[scale=.75]
        \tikzstyle{vertex}=[inner sep=0mm]
        \node (u1) at (0,0) [vertex] {};
        \node (u2) at (2.71,-.71) [vertex] {};
        \node (u3) at (3.1,-3.44) [vertex] {};
        \node (u4) at (.25,-4.44) [vertex] {};
        \node (u5) at (-1.4,-2.34) [vertex] {};
        \node (v1) at (2,0) [vertex] {};
        \node (v2) at (3.4,-2.34) [vertex] {};
        \node (v3) at (1.75,-4.44) [vertex] {};
        \node (v4) at (-1.1,-3.44) [vertex] {};
        \node (v5) at (-.71,-.71) [vertex] {};
        \node (n) at (-8,-2.44) {};
        
        \begin{scope}[fill opacity = .5]
        \filldraw[fill = pink] ($(v1)+(.6,.3)$)
            to[out=-135,in=-45] ($(v1)+(.3,.3)$)
            to[out=135,in=45] ($(v1)+(-.3,.3)$)
            to[out=-135,in=135] ($(v1)+(-.3,-.3)$)
            to[out=-45,in=135] ($(u2)+(-.3,-.3)$)
            to[out=-45,in=-135] ($(u2)+(.3,-.3)$)
            to[out=45,in=-45] ($(u2)+(.3,.3)$)
            to[out=135,in=-135] ($(u2)+(.3,.6)$);
        \filldraw[fill = yellow] ($(u1)+(-.6,.3)$)
            to[in=-135,out=-45] ($(u1)+(-.3,.3)$)
            to[in=135,out=45] ($(u1)+(.3,.3)$)
            to[in=45,out=-45] ($(u1)+(.3,-.3)$)
            to[in=45,out=-135] ($(v5)+(.3,-.3)$)
            to[in=-45,out=-135] ($(v5)+(-.3,-.3)$)
            to[in=-135,out=135] ($(v5)+(-.3,.3)$)
            to[in=-45,out=45] ($(v5)+(-.3,.6)$);
        \filldraw[fill = magenta] ($(u5)+(-.5,-.4)$)
            to[out=15,in=-75] ($(u5)+(-.4,-.1)$)
            to[out=105,in=195] ($(u5)+(-.1,.4)$)
            to[out=15,in=105] ($(u5)+(.4,.1)$)
            to[out=-75,in=105] ($(v4)+(.4,.1)$)
            to[out=-75,in=15] ($(v4)+(.1,-.4)$)
            to[out=195,in=-75] ($(v4)+(-.4,-.1)$)
            to[out=105,in=15] ($(v4)+(-.66,.2)$);
        \filldraw[fill = orange] ($(u4)+(.4,-.7)$)
            to[out=90,in=0] ($(u4)+(0,-.4)$)
            to[out=180,in=270] ($(u4)+(-.4,0)$)
            to[out=90,in=180] ($(u4)+(0,.4)$)
            to[out=0,in=180] ($(v3)+(0,.4)$)
            to[out=0,in=90] ($(v3)+(.4,0)$)
            to[out=270,in=0] ($(v3)+(0,-.4)$)
            to[out=180,in=90] ($(v3)+(-.4,-.7)$);
        \filldraw[semithick] (u1) -- (v1);
        \filldraw[semithick] (u2) -- (v2);
        \filldraw[semithick] (u3) -- (v3);
        \filldraw[semithick] (u4) -- (v4);
        \filldraw[semithick] (u5) -- (v5);
        \end{scope}
        
        \fill (u1) circle (.075);
        \fill (u2) circle (.075);
        \fill (u4) circle (.075);
        \fill (u5) circle (.075);
        \fill (v1) circle (.075);
        \fill (v3) circle (.075);
        \fill (v4) circle (.075);
        \fill (v5) circle (.075);
        \fill (u1) node [above,yshift=.25cm] {$\scaleobj{.75}{u_1}$};
        \fill (u2) node [right,xshift=.25cm] {$\scaleobj{.75}{u_2}$};
        \fill (u4) node [below, xshift=-.2cm, yshift=-.2cm] {$\scaleobj{.75}{u_{\ell-1}}$};
        \fill (u5) node [left,xshift=-.25cm] {$\scaleobj{.75}{u_{\ell}}$};
        \fill (v1) node [above,yshift=.25cm] {$\scaleobj{.75}{v_1}$};
        \fill (v3) node [below, yshift=-.2cm, xshift=.2cm] {$\scaleobj{.75}{v_{\ell-2}}$};
        \fill (v4) node [left,xshift=-.2cm,yshift=-.2cm] {$\scaleobj{.75}{v_{\ell-1}}$};
        \fill (v5) node [left,xshift=-.25cm] {$\scaleobj{.75}{v_{\ell}}$};
        \fill (n) node 
        {
        $\begin{tikzcd}[column sep = small]
        & y_1 \ar[r] & y_2 \ar[dr] & \\
        y_\ell \ar[ur] & & & \text{} \\
        & y_{\ell-1} \ar[ul] & \text{} \ar[l] &
        \end{tikzcd}$
        };
        \fill (n) node [xshift=1.25cm,yshift=-.5cm] {$\iddots$};
        \fill ($(u1)+(1,.3)$) node {$\scaleobj{.75}{E_1}$};
        \fill ($(u2)+(.66,-.65)$) node {$\scaleobj{.75}{E_2}$};
        \fill ($(v5)+(-.66,-.65)$) node {$\scaleobj{.75}{E_{\ell}}$};
        \fill ($(u4)+(-.85,.3)$) node {$\scaleobj{.75}{E_{\ell-1}}$};
        \fill ($(v3)+(1.1,.3)$) node {$\scaleobj{.75}{E_{\ell-2}}$};
        \fill (3.25,-2.89) node {$\cdot$};
        \fill (3.3,-2.71) node {$\cdot$};
        \fill (3.2,-3.07) node {$\cdot$};
    \end{tikzpicture}
    \caption{A reaction cycle and its realization in $\mathcal{H}_N$}
    \label{fig:reaction_cycle}
\end{figure}
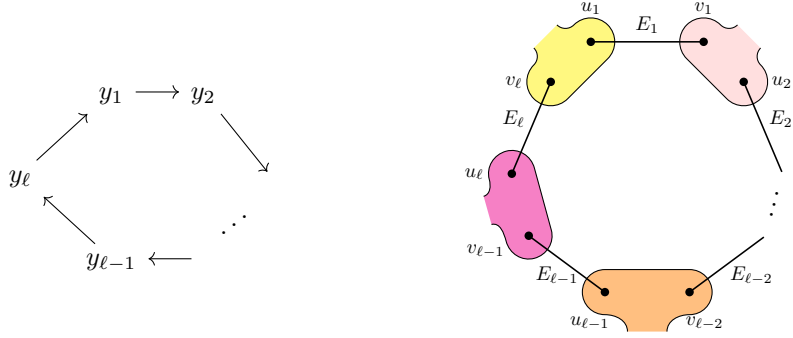

For each color $c \in \{r,b\}$
and vertex $w \in V(\mathcal{H}_N)$, define $\sigma_c(w)$ to be the number of species edges in $\mathscr{E}_c$ that cover $w$. Similarly, define $\rho_c(w)$ to be the number of reaction edges in $\mathscr{E}_c$ that cover $w$. Then the degree of $w$ with respect to $\mathscr{E}_c$ is $\sigma_c(w) + \rho_c(w)$.


Consider the reaction $y_i \to y_{i+1}$ on $C$ where the indices are taken cyclically. Since $y_{i+1}$ is the reactant of the next reaction in the sequence on $C$, then $y_{i+1}$ must be nonempty. Thus, the reaction edge $E_i$ corresponding to $y_i \to y_{i+1}$ is nonempty. Namely, $E_i$ covers both $u_i$ and $v_i$. Necessarily, $\rho_r(u_i) = \rho_r(v_i)$ and $\rho_b(u_i) = \rho_b(v_i)$. This gives rise to the following equation:
\begin{equation}\label{travel_reaction_edge}
    \rho_r(u_i) - \rho_b(u_i) = \rho_r(v_i) - \rho_b(v_i) \; , \quad i = 1, \ldots, \ell
\end{equation}

Let $y_i$ be a complex on $C$. Then $y_i$ is the reactant complex of the reaction $y_i \to y_{i+1}$ and the product complex of the reaction $y_{i-1} \to y_i$. Thus, the vertices $u_i$ and $v_{i-1}$ both correspond to $y_i$. Necessarily, if $s \in \supp{(y_i)}$, then $E_s$ covers both $u_i$ and $v_{i-1}$; see the colored edges in Figure \ref{fig:reaction_cycle}. Thus, $\sigma_r(u_i) = \sigma_r(v_{i-1})$ and $\sigma_b(u_i) = \sigma_b(v_{i-1})$. This gives rise to the following equation:
\begin{equation}\label{travel_species_edge}
    \sigma_b(v_{i-1}) - \sigma_r(v_{i-1}) = \sigma_b(u_{i}) - \sigma_r(u_{i}) \; , \quad i = 1, \ldots, \ell
\end{equation}

Since $u_1$ is almost balanced with respect to $\mathscr{E} = \mathscr{E}_r \sqcup \mathscr{E}_b$, then $\deg_{\mathscr{E}_r}(u_1) = \deg_{\mathscr{E}_b}(u_1) + k$ for some positive integer $k$ and $\deg_{\mathscr{E}_r}(w) = \deg_{\mathscr{E}_b}(w)$ for all $w \in V(\mathcal{H}_N) \setminus \{u_1\}$. For $w \in V(\mathcal{H}_N) \setminus \{u_1\}$, $\sigma_r(w) + \rho_r(w) = \sigma_b(w) + \rho_b(w)$, or equivalently, 
\begin{equation}\label{change_edge}
    \rho_r(w) - \rho_b(w) = \sigma_b(w) - \sigma_r(w) \; , \quad \forall \; w \in V(\mathcal{H}_N) \setminus \{u_1\}
\end{equation}

Consider the vertex $v_i$ for some $i \in \{1, \ldots, \ell\}$. Then 
\begin{align*}
    \rho_r(v_i) - \rho_b(v_i) &\stackrel{\ref{change_edge}}{=} \sigma_b(v_i) - \sigma_r(v_i) \\
    &\stackrel{\ref{travel_species_edge}}{=} \sigma_b(u_{i+1}) - \sigma_r(u_{i+1})\\
    &\stackrel{\ref{change_edge}}{=} \rho_r(u_{i+1}) - \rho_b(u_{i+1}) \\
    &\stackrel{\ref{travel_reaction_edge}}{=} \rho_r(v_{i+1}) - \rho_b(v_{i+1}).
\end{align*}
Written compactly, we have the following equation:
\begin{equation}\label{same_difference}
    \rho_r(v_i) - \rho_b(v_i) = \rho_r(v_{i+1}) - \rho_b(v_{i+1}) \; , \quad i = 1, \ldots, \ell
\end{equation}

We will show $k=0$, a contradiction since $k \in \mathbb{Z}_{>0}$, proving $u_1$ is not almost balanced. We do so by the following sequence of equalities:
\begin{align*}
    k &= \deg_{\mathscr{E}_r}(u_1) - \deg_{\mathscr{E}_b}(u_1) \\
    &= \sigma_r(u_1) + \rho_r(u_1) - \sigma_b(u_1) - \rho_b(u_1) \\
    &= \sigma_r(u_1) - \sigma_b(u_1) + \rho_r(u_1) - \rho_b(u_1) \\
    &\stackrel{\ref{travel_species_edge}}{=} \sigma_r(v_\ell) - \sigma_b(v_\ell) + \rho_r(u_1) - \rho_b(u_1) \\
    &\stackrel{\ref{travel_reaction_edge}}{=} \sigma_r(v_\ell) - \sigma_b(v_\ell) + \rho_r(v_1) - \rho_b(v_1) \\
    &\stackrel{\ref{same_difference}}{=} \sigma_r(v_\ell) - \sigma_b(v_\ell) + \rho_r(v_\ell) - \rho_b(v_\ell) \\
    &= \sigma_r(v_\ell) + \rho_r(v_\ell) - \sigma_b(v_\ell) - \rho_b(v_\ell) \\
    &= \deg_{\mathscr{E}_r}(v_\ell) - \deg_{\mathscr{E}_b}(v_\ell) \\
    &= 0.
\end{align*}
\end{proof}

\begin{corollary}
Let $N$ be a weakly reversible chemical reaction network. No vertex in $V(\mathcal{H}_N)$ is almost balanced.
\end{corollary}

Consider two networks $N = (\mathscr{S}, \mathscr{C}, \mathscr{R})$ and $N' = (\mathscr{S}', \mathscr{C}', \mathscr{R}')$ such that $\mathscr{S} \subseteq \mathscr{S}'$. For complexes $y \in \mathscr{C}$ and $y' \in \mathscr{C}'$, we write $y \mid y'$ if the coefficient of $s$ in $y$ is at most the coefficient of $s$ in $y'$ for each species $s \in \mathscr{S}$. Divides as a relation between complexes matches our intuition of divides in the usual sense: $y \mid y'$ if and only if $x^y \mid x^{y'}$. Similar notation is found in \cite{SS}. A complex $y \in \mathscr{C}$ is said to be \textit{minimal} in the network $N$ if, for any $y' \in \mathscr{C}$, either 1) $y \mid y'$ or 2) $y \nmid y'$ and $y' \nmid y$. 

\begin{theorem}\label{min_react_generate_ideal}
Let $N$ be a 0,1-network and suppose $y_{1}, \ldots, y_{k}$ are the minimal reactants of $N$. If, for all $i = 1, \ldots, k$, either $u_{i}$ or $v_{i}$ in $V(\mathcal{H}_N)$ is almost balanced, then $\mathcal{I}(N)$ is a monomial ideal, and in particular, $\mathcal{I}(N) = \langle x^{y_{1}} , \ldots , x^{y_{k}} \rangle$.
\end{theorem}

\begin{proof} Let $N = (\mathscr{S}, \mathscr{C}, \mathscr{R})$ be a 0,1-network and let $y_{1}, \ldots, y_{k}$ be the minimal reactants of $N$. By Proposition \ref{finestideal}, $\mathcal{I}(N) \subseteq \langle x^{y_i} : y_i \to y_i' \in \mathscr{R} \rangle$. Without loss of generality, suppose the vertices $u_{1}, \ldots,  u_{k}$ in $\mathcal{H}(N)$ are almost balanced. By Corollary \ref{monomial_in_ideal}, $x^{y_{i}} \in \mathcal{I}(N)$ for $i = 1, \ldots, k$. Thus, $\langle x^{y_i}, \ldots, x^{y_k} \rangle \subseteq \mathcal{I}(N) \subseteq \langle x^{y_i} : y_i \to y_i' \in \mathscr{R} \rangle$. We show $\langle x^{y_i} : y_i \to y_i' \in \mathscr{R} \rangle \subseteq \langle x^{y_i}, \ldots, x^{y_k} \rangle$. Let $y \to y'$ be any reaction in $\mathscr{R}$. Since $y_1, \ldots, y_k$ are the minimal reactants of $N$, then there is some $\ell \in \{1, \ldots, k\}$ such that $y_\ell \mid y$. Necessarily, $x^{y_\ell} \mid x^y$ and $y-y_\ell$ has nonnegative coefficients. Since $x^{y_\ell} \in \mathcal{I}(N)$, then $x^y = x^{y_\ell}x^{y-y_\ell} \in \langle x^{y_1}, \ldots, x^{y_k} \rangle$. Thus, $\langle x^{y_i} : y_i \to y_i' \in \mathscr{R} \rangle \subseteq \langle x^{y_i}, \ldots, x^{y_k} \rangle$ and, consequently, $\mathcal{I}(N) = \langle x^{y_1}, \ldots, x^{y_k} \rangle$.
\end{proof}

The converse of the previous theorem does not hold. Consider the following 0,1-network: $$
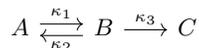
\begin{tikzcd}[column sep = scriptsize] A \ar[r,"\kappa_1",shift left] & B \ar[l,"\kappa_2",shift left] \ar[r,"\kappa_3"] & C \end{tikzcd}$$
This network has minimal reactant complexes $A$ and $B$ and has monomial steady state ideal $\langle x_A , x_B \rangle$. While the minimal reactant $B$ corresponds to an almost balanced vertex in $V(\mathcal{H})$, the minimal reactant $A$ does not. The minimal reactant $A$ does not correspond to an almost balanced vertex by Proposition \ref{reversible_is_balanced}. 

\section{Ideal Preserving Operations}

As seen previously, two different networks can have the same steady state ideal. As another example, Figure \ref{fig:same_ideal_networks} depicts two chemical reaction networks with the same steady state ideal. In this chapter, we discuss three ways that a chemical reaction network can be modified so that the network ideal is preserved. The operations on the network we consider are: adding a single species to the reactant complex of a reaction, adding a single species to the product complex of a reaction, and adding a single reaction.

\begin{figure}[!htb]
    \begin{center}
    \begin{tikzcd}[row sep = tiny, column sep = small]
    N_1 : & A \ar[r] & \varnothing &&& N_2:  & A \ar[r, shift left] & B+C \ar[l, shift left]\\
    & B \ar[r] & \varnothing &&&&  B \ar[r, shorten >= 2ex] & C
    \end{tikzcd}
    \end{center}
    \caption{Two chemical reaction networks with steady state ideal $\langle x_A, x_B \rangle$.}
    \label{fig:same_ideal_networks}
\end{figure}
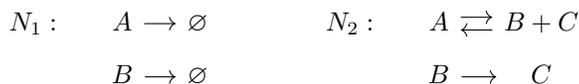

The goal of this chapter is to preserve the steady state ideal of a chemical reaction network, we seek to preserve the generators of the steady state ideal upon performing one of the three previously mentioned operations. In particular, if there are monomial generators in the ideal, then we wish to preserve these monomials. In Chapter \ref{Combinatorics_of_CRNs}, we saw that almost balanced vertices in the network hypergraph imply the existence of monomials in the steady state ideal. Thus, preserving monomials in the steady state ideal can be achieved by preserving almost balanced vertices in the network hypergraph.

First, let us illustrate how changes in the network affect its associated hypergraph. Observe that from Figure \ref{fig:same_ideal_networks}, $N_2$ is obtained from $N_1$ by the following sequence of operations: 1) add $C$ to the product of the reaction $A \to \varnothing$, 2) add $C$ to the product of the reaction $B \to \varnothing$, 3) add $B$ to the product of the reaction $A \to C$, and 4) add the reaction $B+C \to A$. In general, we see that adding to a network will introduce vertices or edges in the hypergraph or cause edges to ``expand'', that is, the edges present in the hypergraph will cover more vertices in the newly obtained hypergraph; see Figure \ref{fig:hypergraph_changes}.

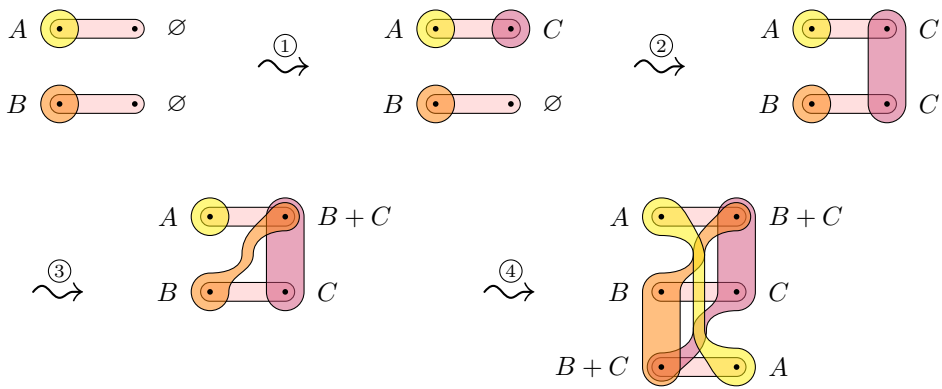
\begin{figure}[!htb]
    \captionsetup{justification=centering,margin=2cm}
    \begin{center}
    \begin{tikzpicture}[scale=.5]
    \node (A1) at (-4,0) {};
    \node (A2) at (6,0) {};
    \node (A3) at (16,0) {};
    \node (A4) at (0,-5) {};
    \node (A5) at (12,-5) {};
    \node (A6) at (14,-9) {};
    \node (B1) at (-4,-2) {};
    \node (B2) at (6,-2) {};
    \node (B3) at (16,-2) {};
    \node (B4) at (0,-7) {};
    \node (B5) at (12,-7) {};
    \node (C1) at (8,0) {};
    \node (C2) at (18,0) {};
    \node (C3) at (18,-2) {};
    \node (C4) at (2,-7) {};
    \node (C5) at (14,-7) {};
    \node (01) at (-2,0) {};
    \node (02) at (-2,-2) {};
    \node (03) at (8,-2) {};
    \node (BC1) at (2,-5) {};
    \node (BC2) at (14,-5) {};
    \node (BC3) at (12,-9) {};
    \node (M) at (1,-6) {};
    \node (TM) at (13,-6) {};
    \node (BM) at (13,-8) {};
    \node (ARR1) at (2,-1) {};
    \node (ARR2) at (12,-1) {};
    \node (ARR3) at (-4,-6) {};
    \node (ARR4) at (8,-7) {};
    
    \begin{scope}[fill opacity=.5]
    \filldraw[fill=pink] 
        ($(A1) + (0,0.25)$)
        to[out=0,in=180] ($(01) + (0,0.25)$)
        to[out=0,in=90] ($(01) + (0.25,0)$)
        to[out=270,in=0] ($(01) + (0,-0.25)$)
        to[out=180,in=0] ($(A1) + (0,-0.25)$)
        to[out=180,in=270] ($(A1) + (-0.25,0)$)
        to[out=90,in=180] ($(A1) + (0,0.25)$);
    \filldraw[fill=pink] 
        ($(B1) + (0,0.25)$)
        to[out=0,in=180] ($(02) + (0,0.25)$)
        to[out=0,in=90] ($(02) + (0.25,0)$)
        to[out=270,in=0] ($(02) + (0,-0.25)$)
        to[out=180,in=0] ($(B1) + (0,-0.25)$)
        to[out=180,in=270] ($(B1) + (-0.25,0)$)
        to[out=90,in=180] ($(B1) + (0,0.25)$);
    \filldraw[fill=yellow] 
    ($(A1) + (0,0.5)$)
        to[out=0,in=90] ($(A1) + (.5,0)$)
        to[out=270,in=0] ($(A1) + (0,-0.5)$)
        to[out=180,in=270] ($(A1) + (-0.5,0)$)
        to[out=90,in=180] ($(A1) + (0,0.5)$);
    \filldraw[fill=orange] 
    ($(B1) + (0,0.5)$)
        to[out=0,in=90] ($(B1) + (.5,0)$)
        to[out=270,in=0] ($(B1) + (0,-0.5)$)
        to[out=180,in=270] ($(B1) + (-0.5,0)$)
        to[out=90,in=180] ($(B1) + (0,0.5)$);
        
    \filldraw[fill=pink] 
        ($(A2) + (0,0.25)$)
        to[out=0,in=180] ($(C1) + (0,0.25)$)
        to[out=0,in=90] ($(C1) + (0.25,0)$)
        to[out=270,in=0] ($(C1) + (0,-0.25)$)
        to[out=180,in=0] ($(A2) + (0,-0.25)$)
        to[out=180,in=270] ($(A2) + (-0.25,0)$)
        to[out=90,in=180] ($(A2) + (0,0.25)$);
    \filldraw[fill=pink] 
        ($(B2) + (0,0.25)$)
        to[out=0,in=180] ($(03) + (0,0.25)$)
        to[out=0,in=90] ($(03) + (0.25,0)$)
        to[out=270,in=0] ($(03) + (0,-0.25)$)
        to[out=180,in=0] ($(B2) + (0,-0.25)$)
        to[out=180,in=270] ($(B2) + (-0.25,0)$)
        to[out=90,in=180] ($(B2) + (0,0.25)$);
    \filldraw[fill=yellow] 
    ($(A2) + (0,0.5)$)
        to[out=0,in=90] ($(A2) + (.5,0)$)
        to[out=270,in=0] ($(A2) + (0,-0.5)$)
        to[out=180,in=270] ($(A2) + (-0.5,0)$)
        to[out=90,in=180] ($(A2) + (0,0.5)$);
    \filldraw[fill=orange] 
    ($(B2) + (0,0.5)$)
        to[out=0,in=90] ($(B2) + (.5,0)$)
        to[out=270,in=0] ($(B2) + (0,-0.5)$)
        to[out=180,in=270] ($(B2) + (-0.5,0)$)
        to[out=90,in=180] ($(B2) + (0,0.5)$);
    \filldraw[fill=purple!70!white] 
    ($(C1) + (0,0.5)$)
        to[out=0,in=90] ($(C1) + (.5,0)$)
        to[out=270,in=0] ($(C1) + (0,-0.5)$)
        to[out=180,in=270] ($(C1) + (-0.5,0)$)
        to[out=90,in=180] ($(C1) + (0,0.5)$);
    
    \filldraw[fill=pink] 
        ($(A3) + (0,0.25)$)
        to[out=0,in=180] ($(C2) + (0,0.25)$)
        to[out=0,in=90] ($(C2) + (0.25,0)$)
        to[out=270,in=0] ($(C2) + (0,-0.25)$)
        to[out=180,in=0] ($(A3) + (0,-0.25)$)
        to[out=180,in=270] ($(A3) + (-0.25,0)$)
        to[out=90,in=180] ($(A3) + (0,0.25)$);
    \filldraw[fill=pink] 
        ($(B3) + (0,0.25)$)
        to[out=0,in=180] ($(C3) + (0,0.25)$)
        to[out=0,in=90] ($(C3) + (0.25,0)$)
        to[out=270,in=0] ($(C3) + (0,-0.25)$)
        to[out=180,in=0] ($(B3) + (0,-0.25)$)
        to[out=180,in=270] ($(B3) + (-0.25,0)$)
        to[out=90,in=180] ($(B3) + (0,0.25)$);
    \filldraw[fill=yellow] 
    ($(A3) + (0,0.5)$)
        to[out=0,in=90] ($(A3) + (.5,0)$)
        to[out=270,in=0] ($(A3) + (0,-0.5)$)
        to[out=180,in=270] ($(A3) + (-0.5,0)$)
        to[out=90,in=180] ($(A3) + (0,0.5)$);
    \filldraw[fill=orange] 
    ($(B3) + (0,0.5)$)
        to[out=0,in=90] ($(B3) + (.5,0)$)
        to[out=270,in=0] ($(B3) + (0,-0.5)$)
        to[out=180,in=270] ($(B3) + (-0.5,0)$)
        to[out=90,in=180] ($(B3) + (0,0.5)$);
    \filldraw[fill=purple!70!white] 
    ($(C2) + (0,0.5)$)
        to[out=0,in=90] ($(C2) + (.5,0)$)
        to[out=270,in=90] ($(C3) + (.5,0)$)
        to[out=270,in=0] ($(C3) + (0,-0.5)$)
        to[out=180,in=270] ($(C3) + (-0.5,0)$)
        to[out=90,in=270] ($(C2)+(-.5,0)$)
        to[out=90,in=180] ($(C2) + (0,0.5)$);
        
    \filldraw[fill=pink] 
        ($(A4) + (0,0.25)$)
        to[out=0,in=180] ($(BC1) + (0,0.25)$)
        to[out=0,in=90] ($(BC1) + (0.25,0)$)
        to[out=270,in=0] ($(BC1) + (0,-0.25)$)
        to[out=180,in=0] ($(A4) + (0,-0.25)$)
        to[out=180,in=270] ($(A4) + (-0.25,0)$)
        to[out=90,in=180] ($(A4) + (0,0.25)$);
    \filldraw[fill=pink] 
        ($(B4) + (0,0.25)$)
        to[out=0,in=180] ($(C4) + (0,0.25)$)
        to[out=0,in=90] ($(C4) + (0.25,0)$)
        to[out=270,in=0] ($(C4) + (0,-0.25)$)
        to[out=180,in=0] ($(B4) + (0,-0.25)$)
        to[out=180,in=270] ($(B4) + (-0.25,0)$)
        to[out=90,in=180] ($(B4) + (0,0.25)$);
    \filldraw[fill=yellow] 
    ($(A4) + (0,0.5)$)
        to[out=0,in=90] ($(A4) + (.5,0)$)
        to[out=270,in=0] ($(A4) + (0,-0.5)$)
        to[out=180,in=270] ($(A4) + (-0.5,0)$)
        to[out=90,in=180] ($(A4) + (0,0.5)$);
    \filldraw[fill=purple!70!white] 
    ($(BC1) + (0,0.5)$)
        to[out=0,in=90] ($(BC1) + (.5,0)$)
        to[out=270,in=90] ($(C4) + (.5,0)$)
        to[out=270,in=0] ($(C4) + (0,-0.5)$)
        to[out=180,in=270] ($(C4) + (-0.5,0)$)
        to[out=90,in=270] ($(BC1)+(-.5,0)$)
        to[out=90,in=180] ($(BC1) + (0,0.5)$);
    \filldraw[fill=orange] 
    ($(B4) + (0,0.5)$)
        to[out=0,in=270] ($(M)+(-.15,0)$)
        to[out=90,in=230] ($(BC1)+(-.375,.2)$)
        to[out=45,in=180] ($(BC1)+(0,.375)$)
        to[out=0,in=90] ($(BC1)+(.375,0)$)
        to[out=270,in=0] ($(BC1)+(0,-.375)$)
        to[out=180,in=90] ($(M)+(.15,0)$)
        to[out=270,in=90] ($(B4) + (.5,0)$)
        to[out=270,in=0] ($(B4)+(0,-.5)$)
        to[out=180,in=270] ($(B4)+(-.5,0)$)
        to[out=90,in=180] ($(B4) + (0,0.5)$);
        
    \filldraw[fill=pink] 
        ($(A5) + (0,0.25)$)
        to[out=0,in=180] ($(BC2) + (0,0.25)$)
        to[out=0,in=90] ($(BC2) + (0.25,0)$)
        to[out=270,in=0] ($(BC2) + (0,-0.25)$)
        to[out=180,in=0] ($(A5) + (0,-0.25)$)
        to[out=180,in=270] ($(A5) + (-0.25,0)$)
        to[out=90,in=180] ($(A5) + (0,0.25)$);
    \filldraw[fill=pink] 
        ($(B5) + (0,0.25)$)
        to[out=0,in=180] ($(C5) + (0,0.25)$)
        to[out=0,in=90] ($(C5) + (0.25,0)$)
        to[out=270,in=0] ($(C5) + (0,-0.25)$)
        to[out=180,in=0] ($(B5) + (0,-0.25)$)
        to[out=180,in=270] ($(B5) + (-0.25,0)$)
        to[out=90,in=180] ($(B5) + (0,0.25)$);
    \filldraw[fill=pink] 
        ($(BC3) + (0,0.25)$)
        to[out=0,in=180] ($(A6) + (0,0.25)$)
        to[out=0,in=90] ($(A6) + (0.25,0)$)
        to[out=270,in=0] ($(A6) + (0,-0.25)$)
        to[out=180,in=0] ($(BC3) + (0,-0.25)$)
        to[out=180,in=270] ($(BC3) + (-0.25,0)$)
        to[out=90,in=180] ($(BC3) + (0,0.25)$);
    \filldraw[fill=purple!70!white] 
    ($(BC2) + (0,0.5)$)
        to[out=0,in=90] ($(BC2) + (.5,0)$)
        to[out=270,in=90] ($(C5)+(.5,0)$)
        to[out=270,in=0] ($(C5) + (0,-0.5)$)
        to[out=180,in=90] ($(BM)+(.2,0)$)
        to[out=270,in=45] ($(BC3)+(.375,-.2)$)
        to[out=230,in=0] ($(BC3)+(0,-.375)$)
        to[out=180,in=270] ($(BC3)+(-.375,0)$)
        to[out=90,in=180] ($(BC3)+(0,.375)$)
        to[out=0,in=270] ($(BM)+(-.15,0)$)
        to[out=90,in=270] ($(C5) + (-0.5,0)$)
        to[out=90,in=270] ($(BC2)+(-.5,0)$)
        to[out=90,in=180] ($(BC2) + (0,0.5)$);
    \filldraw[fill=orange] 
    ($(B5) + (0,0.5)$)
        to[out=0,in=270] ($(TM)+(-.15,0)$)
        to[out=90,in=230] ($(BC2)+(-.375,.2)$)
        to[out=45,in=180] ($(BC2)+(0,.375)$)
        to[out=0,in=90] ($(BC2)+(.375,0)$)
        to[out=270,in=0] ($(BC2)+(0,-.375)$)
        to[out=180,in=90] ($(TM)+(.15,0)$)
        to[out=270,in=90] ($(B5) + (.5,0)$)
        to[out=270,in=90] ($(BC3)+(.5,0)$)
        to[out=270,in=0] ($(BC3) + (0,-0.5)$)
        to[out=180,in=270] ($(BC3) + (-0.5,0)$)
        to[out=90,in=270] ($(B5)+(-.5,0)$)
        to[out=90,in=180] ($(B5) + (0,0.5)$);
    \filldraw[fill=yellow] 
    ($(A5) + (0,0.5)$)
        to[out=0,in=120] ($(A5) + (.7,0)$)
        to[out=300,in=100] ($(TM)+(.15,0)$)
        to[out=270,in=90] ($(BM)+(.15,0)$)
        to[out=270,in=180] ($(A6)+(0,.5)$)
        to[out=0,in=90] ($(A6)+(.5,0)$)
        to[out=270,in=0] ($(A6)+(0,-.5)$)
        to[out=180,in=300] ($(A6)+(-.7,0)$)
        to[out=120,in=270] ($(BM)+(-.15,0)$)
        to[out=90,in=270] ($(TM)+(-.15,0)$)
        to[out=90,in=0] ($(A5) + (0,-0.5)$)
        to[out=180,in=270] ($(A5) + (-0.5,0)$)
        to[out=90,in=180] ($(A5) + (0,0.5)$);
    \end{scope}
    
    \fill (A1) circle (.075);
    \fill (A2) circle (.075);
    \fill (A3) circle (.075);
    \fill (A4) circle (.075);
    \fill (A5) circle (.075);
    \fill (A6) circle (.075);
    \fill (B1) circle (.075);
    \fill (B2) circle (.075);
    \fill (B3) circle (.075);
    \fill (B4) circle (.075);
    \fill (B5) circle (.075);
    \fill (C1) circle (.075);
    \fill (C2) circle (.075);
    \fill (C3) circle (.075);
    \fill (C4) circle (.075);
    \fill (C5) circle (.075);
    \fill (01) circle (.075);
    \fill (02) circle (.075);
    \fill (03) circle (.075);
    \fill (BC1) circle (.075);
    \fill (BC2) circle (.075);
    \fill (BC3) circle (.075);
    
    \fill (ARR1) node [above,yshift=-.3cm] {$\stackrel{\encircled{1}}{\scaleobj{2}{\leadsto}}$};
    \fill (ARR2) node [above,yshift=-.3cm] {$\stackrel{\encircled{2}}{\scaleobj{2}{\leadsto}}$};
    \fill (ARR3) node [above,yshift=-.8cm] {$\stackrel{\encircled{3}}{\scaleobj{2}{\leadsto}}$};
    \fill (ARR4) node [above,yshift=-.3cm] {$\stackrel{\encircled{4}}{\scaleobj{2}{\leadsto}}$};
    
    \fill (A1) node [left,xshift=-.3cm] {$A$};
    \fill (B1) node [left,xshift=-.3cm] {$B$};
    \fill (01) node [right,xshift=.3cm] {$\varnothing$};
    \fill (02) node [right,xshift=.3cm] {$\varnothing$};
    \fill (A2) node [left,xshift=-.3cm] {$A$};
    \fill (B2) node [left,xshift=-.3cm] {$B$};
    \fill (C1) node [right,xshift=.3cm] {$C$};
    \fill (03) node [right,xshift=.3cm] {$\varnothing$};
    \fill (A3) node [left,xshift=-.3cm] {$A$};
    \fill (B3) node [left,xshift=-.3cm] {$B$};
    \fill (C2) node [right,xshift=.3cm] {$C$};
    \fill (C3) node [right,xshift=.3cm] {$C$};
    \fill (A4) node [left,xshift=-.3cm] {$A$};
    \fill (B4) node [left,xshift=-.3cm] {$B$};
    \fill (BC1) node [right,xshift=.3cm] {$B+C$};
    \fill (C4) node [right,xshift=.3cm] {$C$};
    \fill (A5) node [left,xshift=-.3cm] {$A$};
    \fill (B5) node [left,xshift=-.3cm] {$B$};
    \fill (BC3) node [left,xshift=-.3cm] {$B+C$};
    \fill (BC2) node [right,xshift=.3cm] {$B+C$};
    \fill (C5) node [right,xshift=.3cm] {$C$};
    \fill (A6) node [right,xshift=.3cm] {$A$};
    \end{tikzpicture}
    \end{center}
    \caption{A sequence of network hypergraphs after preforming four network operations.}
    \label{fig:hypergraph_changes}
\end{figure}

In each hypergraph from Figure \ref{fig:hypergraph_changes}, the vertex labeled $B$ remains almost balanced: for the first three hypergraphs, $E_B$ is a singleton so take $\mathscr{E}_r = \{E_B\}$ and $\mathscr{E}_b = \varnothing$, for the fourth hypergraph, take $\mathscr{E}_r = \{E_A,E_B\}$ and $\mathscr{E}_b = \{E_1\}$ where $E_1$ corresponds to the reaction $A \to B+C$, and for the fifth hypergraph, take $\mathscr{E}_r = \{E_A,E_B\}$ and $\mathscr{E}_b = \{E_1,E_3\}$ where $E_3$ corresponds to the reaction $B+C \to A$. Therefore, the sequence of operations preformed on $N_1$ preserves the fact that $B$ is almost balanced. Hence, the monomial $x_B$ is in the steady state ideal of each corresponding network. The same, however, cannot be said about $x_A$. The vertex whose label is $A$ only remains almost balanced after the first three operations since $E_A$ is a singleton but is not almost balanced after the fourth operation. Thus, $x_A$ is contained in the steady state ideal of $N_2$ for other reasons that cannot be detected by the almost balanced condition. Indeed, Proposition \ref{reversible_is_balanced} confirms that the vertices covered by $E_1$ corresponding to the complex $A$ and $B+C$ are not almost balanced in the network hypergraph associated to $N_2$. 

The operations we consider are always done reaction-wise. Thus, we need to discuss how to add and remove reactions. Removing the reaction $y \to y'$ belonging to the network $N$ will be denoted $N \setminus \{y \to y'\}$. Adding any reaction $y \to y'$ to the network $N$ will be denoted $N \cup \{y \to y'\}$. From the definition of a chemical reaction network, every complex must belong to some reaction. Then if deleting the reaction $y \to y'$ from $N$ results to either $y$ or $y'$ no longer belonging to some reaction, then they too must be deleted in the process. On the other hand, when adding an arbitrary reaction $y \to y'$ to the network $N$, it is possible that neither $y$ nor $y'$ are complexes in the original network $N$. In such a case, both the reaction and complexes are added to the network. We can think of adding a species $s$ to a reactant complex as deleting and then adding a reaction, for example, adding a species $s$ to the reactant complex $y$ belonging to the reaction $y \to y'$ can now be denoted by $(N \setminus \{y \to y'\}) \cup \{y+s \to y'\}$. As seen in Figure \ref{fig:add_species_to_reactant}, there are two cases to consider: 1) $y$ belongs to one reaction or 2) $y$ belongs to more than one reaction. The main distinction with the two cases is that if $y$ belongs to only one reaction, then $y$ must be deleted and replaced with $y+s$, meaning $y$ by itself no longer appears as a complex in the network, whereas if $y$ belongs to more than one reaction then $y$ will remain as a complex in the network. Also, note that in either of the cases above, $y+s$ may or may not have been a complex in the network prior to adding the species $s$. If $y+s$ is already a complex, then by adding $s$ to $y$, we do not require adding a new complex to the network, only a new reaction. If not, then $y+s$ must be added to the set of complexes. The addition of a species to a product complex is done similarly.

\begin{figure}[!htb]
    \begin{center}
    \begin{tikzpicture}[scale=.75]
    \node (n1) at (0,0) {};
    \node (n2) at (8,0) {};
    \node (n3) at (0,-3) {};
    \node (n4) at (8,-3) {};
    \node (a1) at (4,0) {};
    \node (a2) at (4,-3) {};
    \node (t) at (4,-1.5) {};
    
    \fill (n1) node {
    $\begin{tikzcd}[column sep = small,row sep = small] & & \ar[ld]\\
    y \ar[r] & y' \ar[r] & \text{}\\
    \end{tikzcd}$
    };
    \fill (n2) node {
    $\begin{tikzcd}[column sep = small,row sep = small] \ar[rd] & & & \ar[ld]\\
    \text{} & y+s \ar[r] \ar[l] & y' \ar[r] & \text{}
    \end{tikzcd}$
    };
    \fill (n3) node {
    $\begin{tikzcd}[column sep = small,row sep = small] \ar[rd] & & & \ar[ld]\\
    \text{} & y \ar[r, shift left] \ar[l]  & y' \ar[l, shift left] \ar[r] & \text{}\\
    \end{tikzcd}$
    };
    \fill (n4) node[yshift=-.5cm] {
    $\begin{tikzcd}[column sep = small,row sep = small] \ar[rd] & & & \ar[ld]\\
    \text{} & y \ar[l] & y' \ar[l] \ar[r] & \text{}\\
    \text{} \ar[r] & y+s \ar[ru] \ar[ld] & & \\
    \text{} & & &
    \end{tikzcd}$
    };
    \fill (a1) node {$\scaleobj{2}{\leadsto}$};
    \fill (a2) node {$\scaleobj{2}{\leadsto}$};
    \fill (t) node {Add $s$};
    \fill ($(n1)+(-3,0)$) node {\textit{Case 1:}};
    \fill ($(n3)+(-3,0)$) node {\textit{Case 2:}};
    \end{tikzpicture}
    \end{center}
    \caption{Examples of the addition of species $s$ to the complex $y$.}
    \label{fig:add_species_to_reactant}
\end{figure}
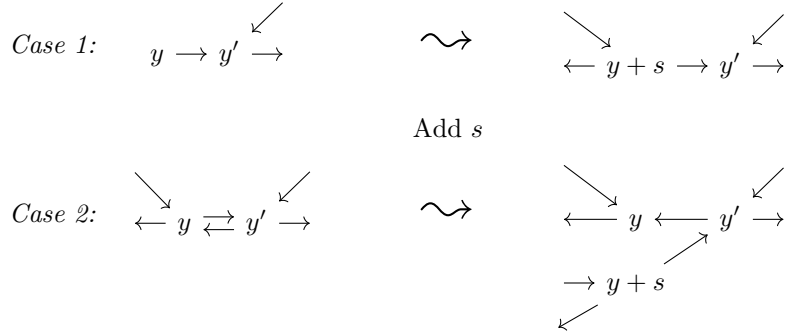

\begin{rmk} \label{rmk:equality}
When obtaining a network $N' = (\mathscr{S}', \mathscr{C}', \mathscr{R}')$ from a network $N = (\mathscr{S}, \mathscr{C}, \mathscr{R})$, it is possible that $N'$ contains a species that $N$ does not, that is, $\mathscr{S} \subsetneq \mathscr{S}'$. Thus, $\mathcal{I}(N)$ and $\mathcal{I}(N')$ can reside in different rings. Namely, $\mathcal{I}(N)$ is an ideal of the ring $\mathbb{Q}(\kappa_1, \ldots, \kappa_{n_\mathscr{R}})[x_1, \ldots, x_{n_\mathscr{S}}]$ and, if $N'$ is obtained from $N$ by adding a species to a complex in the $i$th reaction, then $\mathcal{I}(N')$ is an ideal of the ring $\mathbb{Q}(\kappa_1, \ldots, \kappa_{i-1}, \kappa_{i+1}, \ldots, \kappa_{n_\mathscr{R}}, \kappa_i')[x_1, \ldots, x_{n_\mathscr{S'}}]$ where $\kappa_i'$ is the rate constant for the newly added reaction. Otherwise, if $N'$ is obtained from $N$ be adding a reaction, then $\mathcal{I}(N') \subseteq \mathbb{Q}(\kappa_1, \ldots, \kappa_{n_\mathscr{R}}, \kappa_i')[x_1, \ldots, x_{n_\mathscr{S'}}]$. In either case, when we write $\mathcal{I}(N') = \mathcal{I}(N)$, we mean they are equivalent in the larger ring $
\mathbb{Q}(\kappa_1, \ldots, \kappa_{n_\mathscr{R}}, \kappa_i')[x_1, \ldots, x_{n_\mathscr{S'}}]$.
\end{rmk}

We now present the first ideal preserving operation of the paper.


\begin{theorem}\label{addproduct}(Adding species to a product complex)
Let $N$ be a 0,1-network and suppose $u_{i} \in V(\mathcal{H}_N)$ is almost balanced with respect to the 2-colored multiset $\mathscr{E} = \mathscr{E}_r \sqcup \mathscr{E}_b$ over the edges in $E(\mathcal{H}_N)$. Suppose $s_k$ is a species such that $s_k \not \in \mathrm{supp}(y_i')$, and let $N' = (N \setminus \{y_i \to y_i'\}) \cup \{y_i \to y_i' + s_k\}$.  If $E_{s_k} \not \in \mathscr{E}$, then $\mathcal{I}(N') = \mathcal{I}(N)$.
\end{theorem}

\begin{proof}
Let $N = (\mathscr{S}, \mathscr{C}, \mathscr{R})$ be a 0,1-network and let $u_{i} \in V(\mathcal{H}_N)$ be almost balanced with respect to the 2-colored multiset $\mathscr{E} = \mathscr{E}_r \sqcup \mathscr{E}_b$ over the edges in $E(\mathcal{H}_N)$. Suppose $s_k$ is a species such that $s_k \not \in \mathrm{supp}(y_i')$ and $E_{s_k} \not \in \mathscr{E}$. Obtain a network $N' = (\mathscr{S}', \mathscr{C}', \mathscr{R}')$ from $N$ by removing the reaction $\begin{tikzcd}[column sep = small] y_i \ar[r,"\kappa_{i}"] & y_i' \end{tikzcd}$ and adding the reaction $\begin{tikzcd}[column sep = small] y_i \ar[r,"\kappa_{i}'"] & y_i'+s_k \end{tikzcd}$. Since $s_k \not \in \mathrm{supp}(y_i')$, then the coefficient of $s_k$ in the complex $y_i'+s_k$ is 1. Hence, $N'$ is a 0,1-network.

The modification of $N$ by a single species implies that $\mathscr{S} \subseteq \mathscr{S}'$. For $\ell = 1, \ldots, n_\mathscr{S}$, let $\dot{x}_{s_\ell}$ denote the steady state polynomial for species $s_\ell$ in $\mathscr{S}$ and, for $\ell = 1, \ldots, n_{\mathscr{S}'}$, let $\dot{x}_{s_\ell}'$ denote the steady state polynomial for species $s_\ell$ in $\mathscr{S}'$. To prove $\mathcal{I}(N') = \mathcal{I}(N)$ it is sufficient to show $\dot{x}_{s_\ell}' \in \mathcal{I}(N)$ for all $1 \leq \ell \leq n_{\mathscr{S}'}$ and $\dot{x}_{s_\ell} \in \mathcal{I}(N')$ for all $1 \leq \ell \leq n_\mathscr{S}$. Note that if $y_j \to y_j'$ is a reaction in $\mathscr{R}$ different from $y_i \to y_i'$, then $y_j \to y_j'$ is also in $\mathscr{R}'$, and the stochiometric coefficients of the species do not change for $y_j \to y_j'$
in $N'$. On the other hand, the stoichiometric coefficient of the species $s_k$ in the reaction $y_i \to y_i' + s_k$ has increased by one. Therefore,  $\dot{x}_s' = \dot{x}_s$ for any $s \in \mathscr{S} \setminus \{s_k\}$ and $\dot{x}_{s_k}' = \dot{x}_{s_k} + \kappa_{i}'x^{y_i}$ where $\dot{x}_{s_k} = 0$ if $s_k \not \in \mathscr{S}$. Hence,  $\dot{x}_s' \in \mathcal{I}(N)$ and $\dot{x}_s \in \mathcal{I}(N')$ for any $s \in \mathscr{S} \setminus \{s_k\}$. 

It remains to show that $\dot{x}_{s_k}' \in \mathcal{I}(N)$ and $\dot{x}_{s_k} \in \mathcal{I}(N')$. Since $u_{i}$ is almost balanced, then $x^{y_i} \in \mathcal{I}(N)$ by Corollary \ref{monomial_in_ideal}. Then $\dot{x}_{s_k}' = \dot{x}_{s_k} + \kappa_{i}'x^{y_i} \in \mathcal{I}(N)$. 

To show $\dot{x}_{s_k} \in \mathcal{I}(N')$, notice that $\mathcal{H}_{N}$ and $\mathcal{H}_{N'}$ have the same vertex set, since $N$ and $N'$ have the same number of reactions, moreover, the only difference between $E(\mathcal{H}_{N})$ and $E(\mathcal{H}_{N'})$ is that $E_{s_k}$ is extended to include $v_i$. Since $E_{s_k} \not \in \mathscr{E}$, then $u_{i} \in V(\mathcal{H}_{N'})$, the vertex corresponding to the reactant complex of $y_i \to y_i'+s_k$, is almost balanced in $\mathcal{H}_{N'}$ with respect to $\mathscr{E}' = \mathscr{E}_r' \sqcup \mathscr{E}_b'$ where $\mathscr{E}_r'= \mathscr{E}_r$ but as a multiset over the edges in $E(\mathcal{H}_{N'})$ and $\mathscr{E}_b'= \mathscr{E}_b$ but as a multiset over the edges in $E(\mathcal{H}_{N'})$.  Thus, by Corollary \ref{monomial_in_ideal}, $x^{y_i} \in \mathcal{I}(N')$ and, hence, $\dot{x}_{s_k} = \dot{x}_{s_k}' - \kappa_{i}'x^{y_i} \in \mathcal{I}(N')$. Therefore, $\dot{x}_{\ell}' \in \mathcal{I}(N)$ for all $1 \leq \ell \leq n_{\mathscr{S}'}$ and $\dot{x}_\ell \in \mathcal{I}(N')$ for all $1 \leq \ell \leq n_\mathscr{S}$ or, equivalently, $\mathcal{I}(N') = \mathcal{I}(N)$.
\end{proof}

\begin{example}\label{example1}
Consider the running example in Figure \ref{fig:network} to demonstrate how adding a species to the product of a reaction under the conditions described in Theorem \ref{addproduct} preserves the steady state ideal but changes the network hypergraph. Reaction edges are draw as line segments in the network hypergraph to emphasize their distinction from the species edges. 

By Proposition \ref{finestideal}, $\mathcal{I}(N) \subseteq \langle x_A , x_B \rangle$. Indeed, $\mathcal{I}(N)$ equals $\langle x_A, x_B \rangle$ but requires verification. The vertex labeled $A$ that is incident to $E_3$ is almost balanced with respect to the multiset $\mathscr{E} = \mathscr{E}_r \sqcup \mathscr{E}_b$ over the edges in $E(\mathcal{H}_N)$ where $\mathscr{E}_r = \{E_3\}$ and $\mathscr{E}_b = \{E_C\}$. By Corollary \ref{monomial_in_ideal}, $x_A \in \mathcal{I}(N)$. Further, notice that $x_B = \frac{\kappa_1}{\kappa_2\kappa_3}\dot{x}_C - \frac{1}{\kappa_2}\dot{x}_B$ so $x_B \in \mathcal{I}(N)$. Hence, $\mathcal{I}(N) = \langle x_A,x_B \rangle$. Since $E_B \not \in \mathscr{E}$, then $\mathcal{I}(N_1) = \mathcal{I}(N)$ where $N_1 = (N \setminus \{A \to C\}) \cup \{A \to B+C\}$; see Figure \ref{fig:network1}. 

\begin{figure}[h]
    \captionsetup{justification=centering,margin=2cm}
    \begin{center}
    \begin{tikzpicture}[scale=.75]
\tikzstyle{vertex}=[draw, circle, inner sep=0mm]
    \node (n) at (-6,-1.5) {};
    \node (a1) at (0,0) [vertex] {};
    \node (a2) at (3,-3) [vertex] {};
    \node (a3) at (-1.5,-1.5) [vertex] {};
    \node (b1) at (1.5,0) [vertex] {};
    \node (b2) at (3,-1.5) [vertex] {};
    \node (c) at (-1.5,-3) [vertex] {};
    \node (m) at (.75,-2.25) {};
    \node (poly) at (4,0) {};
    
    \begin{scope}[fill opacity=.5]
    \filldraw[fill = blue!50]
        ($(c)+(0,.25)$) 
        to[out=0,in=90] ($(c)+(.25,0)$)
        to[out=270,in=0] ($(c)+(0,-.25)$)
        to[out=180,in=270] ($(c)+(-.25,0)$)
        to[out=90,in=180] ($(c)+(0,.25)$);
        
    \filldraw[fill = orange!75]
        ($(b1)+(.2,.2)$) 
        to[out=-45,in=135] ($(b2)+(.2,.2)$)
        to[out=-45,in=45] ($(b2)+(.2,-.2)$)
        to[out=-135,in=-45] ($(b2)+(-.2,-.2)$)
        to[out=135,in=-45] ($(b1)+(-.2,-.2)$)
        to[out=135,in=-135] ($(b1)+(-.2,.2)$)
        to[out=45,in=135] ($(b1)+(.2,.2)$);
        
    \filldraw[fill = yellow!75]
        ($(a1)+(.2,.2)$)
        to[out=-45,in=135] ($(m)+(.15,.15)$)
        to[out=-45,in=135] ($(a2)+(.2,.2)$)
        to[out=-45,in=45] ($(a2)+(.2,-.2)$)
        to[out=-135,in=-45] ($(a2)+(-.2,-.2)$)
        to[out=135,in=-45] ($(m)+(-.15,-.15)$)
        to[out=135,in=-45] ($(a3)+(-.2,-.2)$)
        to[out=135,in=-135] ($(a3)+(-.2,.2)$)
        to[out=45,in=-135] ($(a1)+(-.2,.2)$)
        to[out=45,in=135] ($(a1)+(.2,.2)$);
    \end{scope}
    
    \draw[line width=.75pt] (a1) -- (b1);
    \draw[line width=.75pt] (b2) -- (a2);
    \draw[color=red,line width=1pt] (a3) -- (c);
    
    \fill (a1) circle (.075);
    \fill (a2) circle (.075);
    \fill (a3) circle (.075);
    \fill (b1) circle (.075);
    \fill (b2) circle (.075);
    \fill (c) circle (.075);
    
    \fill (a1) node [left,yshift=.5cm] {$\scaleobj{.75}{A}$};
    \fill (a2) node [below,xshift=.5cm] {$\scaleobj{.75}{A}$};
    \fill (a3) node [above,xshift=-.5cm] {$\scaleobj{.75}{A}$};
    \fill (b1) node [right,yshift=.5cm] {$\scaleobj{.75}{B}$};
    \fill (b2) node [above,xshift=.5cm] {$\scaleobj{.75}{B}$};
    \fill (c) node [below,xshift=-.5cm] {$\scaleobj{.75}{C}$};
    \fill (n) node [scale=.85,yshift=.5cm] {
    $\begin{tikzcd}[column sep = scriptsize, row sep = scriptsize] 
    A \ar[r,shift left,"\kappa_1"] \ar[d,"\kappa_3"'] & B \ar[l,shift left,"\kappa_2"]\\
    C &
    \end{tikzcd}$
    };
    \fill ($(m)+(-1.25,1.25)$) node {$\scaleobj{.75}{E_A}$};
    \fill ($(b1)+(.7,-.65)$) node {$\scaleobj{.75}{E_B}$};
    \fill ($(c)+(1,0)$) node {$\scaleobj{.75}{E_C}$};
    \fill ($(c)+(-.3,.7)$) node {$\scaleobj{.75}{E_3}$};
    \fill ($(a2)+(.3,.7)$) node {$\scaleobj{.75}{E_2}$};
    \fill ($(a1)+(.7,.3)$) node {$\scaleobj{.75}{E_1}$};
    \fill ($(a1)+(-8,0)$) node {$\scaleobj{.75}{N:}$};
    \fill ($(a1)+(-3,0)$) node {$\scaleobj{.75}{\mathcal{H}_N:}$};
    
    \fill (poly) node [right, xshift=.75cm] {$\scaleobj{.75}{\dot{x}_A = -(\kappa_1+\kappa_3)x_A + \kappa_2x_B}$};
    \fill ($(poly)+(0,-1.5)$) node [right,xshift=.75cm] {$\scaleobj{.75}{\dot{x}_B = \kappa_1x_A - \kappa_2x_B}$};
    \fill ($(poly)+(0,-3)$) node [right,xshift=.75cm] {$\scaleobj{.75}{\dot{x}_C = \kappa_3x_A}$};
    
    \draw ($(c)+(.15,0)$) -- ($(c)+(.65,0)$);
\end{tikzpicture}
    \end{center}
    \caption{A network $N$, whose steady state ideal is $\langle x_A, x_B \rangle$, and its network hypergraph.}
    \label{fig:network}
\end{figure}
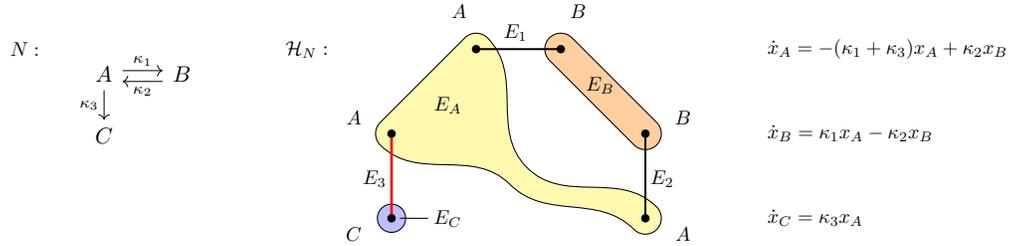

Notice that the vertex labeled $A$ in $V(\mathcal{H}_{N_1})$ that is incident to $E_3 \in E(\mathcal{H}_{N_1})$ is almost balanced with respect to $\mathscr{E}' = \mathscr{E}_r' \sqcup \mathscr{E}_b'$ as a multiset over the edges in $E(\mathcal{H}_{N_1})$ where $\mathscr{E}_r' = \{E_3\}$ and $\mathscr{E}_b' = \{E_C\}$. Thus, $x_A \in \mathcal{I}(N_1)$. Also, $x_B \in \mathcal{I}(N_1)$ since $x_B = \left(\frac{\kappa_1}{\kappa_2\kappa_3'} - \frac{1}{\kappa_2}\right)\dot{x}_C^{(1)} - \frac{1}{\kappa_2}\dot{x}_B^{(1)}$ where $\dot{x}_B^{(1)}$ and $\dot{x}_C^{(1)}$ are the steady state polynomials associated to $N_1$.

\begin{figure}[h]
    \captionsetup{justification=centering,margin=2cm}
    \begin{center}
    \begin{tikzpicture}[scale=.75]
\tikzstyle{vertex}=[inner sep=0mm]
    \node (n) at (-6,-1) {};
    \node (a1) at (0,0) [vertex] {};
    \node (a2) at (3,-3) [vertex] {};
    \node (a3) at (-1.5,-1.5) [vertex] {};
    \node (b1) at (1.5,0) [vertex] {};
    \node (b2) at (3,-1.5) [vertex] {};
    \node (c) at (-1.5,-3) [vertex] {};
    \node (m) at (.75,-2.25) {};
    \node (poly) at (5,0) {};
    
    \begin{scope}[fill opacity=.5]
    \filldraw[fill = orange!75]
        ($(b1)+(.2,.2)$) 
        to[out=-45,in=135] ($(b2)+(.2,.2)$)
        to[out=-45,in=45] ($(b2)+(.2,-.2)$)
        to[out=-135,in=45] ($(m)+(.15,-.15)$)
        to[out=-135,in=45] ($(c)+(.2,-.1)$)
        to[out=-135,in=-45] ($(c)+(-.2,-.1)$)
        to[out=135,in=-135] ($(c)+(-.2,.3)$)
        to[out=45,in=-135] ($(m)+(-.15,.15)$)
        to[out=45,in=-135] ($(b1)+(-.2,.2)$)
        to[out=45,in=135] ($(b1)+(.2,.2)$);
        
    \filldraw[fill = yellow!75]
        ($(a1)+(.2,.2)$)
        to[out=-45,in=135] ($(m)+(.15,.15)$)
        to[out=-45,in=135] ($(a2)+(.2,.2)$)
        to[out=-45,in=45] ($(a2)+(.2,-.2)$)
        to[out=-135,in=-45] ($(a2)+(-.2,-.2)$)
        to[out=135,in=-45] ($(m)+(-.15,-.15)$)
        to[out=135,in=-45] ($(a3)+(-.2,-.2)$)
        to[out=135,in=-135] ($(a3)+(-.2,.2)$)
        to[out=45,in=-135] ($(a1)+(-.2,.2)$)
        to[out=45,in=135] ($(a1)+(.2,.2)$);
        
    \filldraw[fill = blue!50]
        ($(c)+(.1,.2)$) 
        to[out=0,in=90] ($(c)+(.4,-.1)$)
        to[out=270,in=0] ($(c)+(.1,-.4)$)
        to[out=180,in=270] ($(c)+(-.2,-.1)$)
        to[out=90,in=180] ($(c)+(.1,.2)$);
    \end{scope}
    
    \draw[line width=.75pt] (a1) -- (b1);
    \draw[line width=.75pt] (b2) -- (a2);
    \draw[color=red,line width=1pt] (a3) -- (c);
    
    \fill (a1) circle (.075);
    \fill (a2) circle (.075);
    \fill (a3) circle (.075);
    \fill (b1) circle (.075);
    \fill (b2) circle (.075);
    \fill (c) circle (.075);
    
    \fill (a1) node [left,yshift=.5cm] {$\scaleobj{.75}{A}$};
    \fill (a2) node [below,xshift=.5cm] {$\scaleobj{.75}{A}$};
    \fill (a3) node [above,xshift=-.5cm] {$\scaleobj{.75}{A}$};
    \fill (b1) node [right,yshift=.5cm] {$\scaleobj{.75}{B}$};
    \fill (b2) node [above,xshift=.5cm] {$\scaleobj{.75}{B}$};
    \fill (c) node [below,xshift=-1cm] {$\scaleobj{.75}{B+C}$};
    \fill (n) node [scale=.85] {
    $\begin{tikzcd}[column sep = small, row sep = normal] 
    A \ar[r,shift left,"\kappa_1"] \ar[d,"\kappa_3'"'] & B \ar[l,shift left,"\kappa_2"]\\
    B+C &
    \end{tikzcd}$
    };
    \fill ($(m)+(-1.25,1.25)$) node {$\scaleobj{.75}{E_A}$};
    \fill ($(m)+(1.25,1.25)$) node {$\scaleobj{.75}{E_B}$};
    \fill ($(c)+(1,-.25)$) node {$\scaleobj{.75}{E_C}$};
    \fill ($(c)+(-.3,.7)$) node {$\scaleobj{.75}{E_3}$};
    \fill ($(a2)+(.3,.7)$) node {$\scaleobj{.75}{E_2}$};
    \fill ($(a1)+(.7,.3)$) node {$\scaleobj{.75}{E_1}$};
    \fill ($(n)+(-2,1)$) node {$\scaleobj{.75}{N_1:}$};
    \fill ($(a1)+(-3,0)$) node {$\scaleobj{.75}{\mathcal{H}_{N_1}:}$};
    
    \draw ($(c)+(.25,-.2)$) -- ($(c)+(.75,-.2)$);
    
    \fill (poly) node [right] {$\scaleobj{.75}{\dot{x}_A^{(1)} = -(\kappa_1+\kappa_3')x_A + \kappa_2x_B}$};
    \fill ($(poly)+(0,-1.5)$) node [right] {$\scaleobj{.75}{\dot{x}_B^{(1)} = (\kappa_1+\kappa_3')x_A - \kappa_2x_B}$};
    \fill ($(poly)+(0,-3)$) node [right] {$\scaleobj{.75}{\dot{x}_C^{(1)} = \kappa_3'x_A}$};
\end{tikzpicture}
    \end{center}
    \caption{A network $N_1$, obtained from $N$ in Figure \ref{fig:network} by adding $B$ to the product complex of the reaction $A \to C$, and its network hypergraph.}
    \label{fig:network1}
\end{figure}
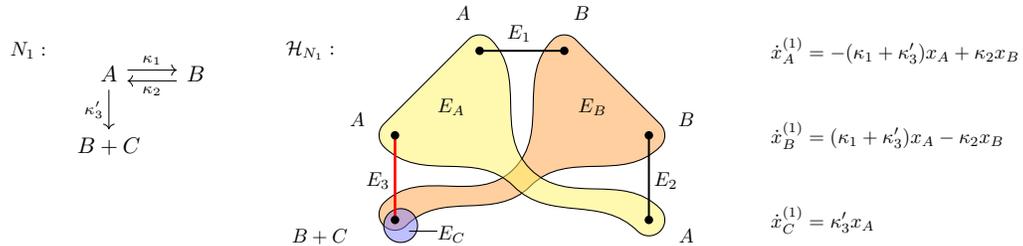
\end{example}

\begin{theorem}\label{addreactant} (Adding species to a reactant complex)
Let $N$ be a 0,1-network. Suppose there are distinct reactions $y_i \to y_i'$ and $y_j \to y_j'$ such that $y_j \mid y_i$ and $u_j$ is almost balanced with respect to some 2-colored edgeset $\mathscr{E} = \mathscr{E}_r \sqcup \mathscr{E}_b$. Suppose $s_k$ is a species such that $s_k \not \in \mathrm{supp}(y_i)$, and let $N' = (N \setminus \{y_i \to y_i'\}) \cup \{y_i+s_k \to y_i'\}$. If $E_{s_k} \not \in \mathscr{E}$, then $\mathcal{I}(N') = \mathcal{I}(N)$.
\end{theorem}

\begin{proof}
Let $N = (\mathscr{S}, \mathscr{C}, \mathscr{R})$ be a 0,1-network and let $y_i \to y_i'$ and $y_j \to y_j'$ be distinct reactions in $\mathscr{R}$ such that $y_j \mid y_i$ and $u_j$ is almost balanced with respect to the 2-colored multiset $\mathscr{E} = \mathscr{E}_r \sqcup \mathscr{E}_b$ over the edges in $E(\mathcal{H}_N)$. Suppose $s_k$ is a species, not necessarily in $\mathscr{S}$, such that $s_k \not \in \mathrm{supp}(y_i)$ and $E_{s_k} \not \in \mathscr{E}$. Obtain a network $N' = (\mathscr{S}', \mathscr{C}', \mathscr{R}')$ from $N$ by removing the reaction $y_i \to y_i'$ and adding the reaction $y_i+s_k \to y_i'$. Since $s_k \not \in \supp{y_i}$, then the coefficient of $s_k$ in $y_i+s_k$ is 1, so $N'$ is a 0,1-network.

Let $S_1' = \supp{(y_i+s_k)} \cup \supp{y_i'}$ and $S_2' = \mathscr{S}' \setminus S_1'$. If $s \in {S}_2'$, then $s$ is not involved in either of the reactions $y_i \to y_i'$ or $y_i+s_k \to y_i'$. Necessarily, the rate at which the concentration of $s$ changes in $N$ remains the same as in $N'$, that is, $\dot{x}_s' = \dot{x}_s$. Thus, $\dot{x}_s' \in \mathcal{I}(N)$ and $\dot{x}_s \in \mathcal{I}(N')$ for all $s \in S_2'$. 

On the other hand, if $s \in S_1'$, then $\dot{x}_s'$ and $\dot{x}_s$ differ by the rate at which the concentration of $s$ changes by the reaction $y_i \to y_i'$ and the rate at which the concentration $s$ changes by the reaction $y_i+s_k \to y_i'$. Specifically, $\dot{x}_s' - \gamma_{s_i}'\kappa_i'x^{y_i+s_k} = \dot{x}_s - \gamma_{s_i}\kappa_ix^{y_i}$. Since $u_j \in V(\mathcal{H}_N)$ is almost balanced, then $x^{y_j} \in \mathcal{I}(N)$. Since $y_j \mid y_i$, then $y_i-y_j$ is nonnegative and $x^{y_i} = x^{y_j}x^{y_i-y_j} \in \mathcal{I}(N)$ and so $x^{y_i+s_k} = x^{y_i}x^{s_k} \in \mathcal{I}(N)$. Thus, $\dot{x}_s' = \dot{x}_s - \gamma_{s_i}\kappa_ix^{y_i} + \gamma_{s_i}'\kappa_i'x^{y_i+s_k} \in \mathcal{I}(N)$. 

Now notice that $\mathcal{H}_{N}$ and $\mathcal{H}_{N'}$ have the same vertex set, moreover, the only difference between $E(\mathcal{H}_{N})$ and $E(\mathcal{H}_{N'})$ is that $E_{s_k}$ is extended to include $u_i$. Since $E_{s_k} \not \in \mathscr{E}$, then $u_{j} \in V(\mathcal{H}_{N'})$, the vertex corresponding to the reactant complex of $y_j \to y_j'$, is almost balanced in $\mathcal{H}_{N'}$ with respect to $\mathscr{E}' = \mathscr{E}_r' \sqcup \mathscr{E}_b'$ where $\mathscr{E}_r'= \mathscr{E}_r$ but as a multiset over the edges in $E(\mathcal{H}_{N'})$ and $\mathscr{E}_b'= \mathscr{E}_b$ but as a multiset over the edges in $E(\mathcal{H}_{N'})$. By Corollary \ref{monomial_in_ideal}, $x^{y_j} \in\mathcal{I}(N')$. Since $x^{y_i} = x^{y_j}x^{y_i-y_j} \in \mathcal{I}(N')$, then we have $\dot{x}_s = \dot{x}_s' - \gamma_{s_i}'\kappa_i'x^{y_i+s_k} + \gamma_{s_i}\kappa_ix^{y_i} \in \mathcal{I}(N')$. Thus, $\dot{x}_s \in \mathcal{I}(N')$ and $\dot{x}_s' \in \mathcal{I}(N)$ for all $s \in S_1'$. Therefore, $\mathcal{I}(N') = \mathcal{I}(N)$.
\end{proof}

\begin{example}\label{example2}
In Figure \ref{fig:network1}, the reactions $A \to B$ and $A \to B+C$ are distinct and certainly $A \mid A$. Since the vertex $u_3$, labeled $A$ and incident to $E_3$, is almost balanced in $\mathcal{H}_{N_1}$, then obtaining a network $N_2$ from $N_1$ by adding $B$ to the reactant of $A \to B$ preserves the steady state ideal; see Figure \ref{fig:network2}. The vertex $u_3 \in V(\mathcal{H}_{N_2})$, labeled $A$ and incident to $E_3$, remains almost balanced in $\mathcal{H}_{N_2}$ with the same colored edges as for $u_3 \in V(\mathcal{H}_{N_1})$, so $x_A \in \mathcal{I}(N_2)$. Notice the vertex $u_2 \in V(\mathcal{H}_{N_1})$, labeled $B$ and incident to $E_2$, is almost balanced in $\mathcal{H}_{N_2}$. To be specific, $u_2$ is almost balanced in $\mathcal{H}_{N_2}$ with respect to the 2-colored edgeset $\mathscr{E} = \mathscr{E}_r \sqcup \mathscr{E}_b$ over $E(\mathcal{H}_{N_2})$ where $\mathscr{E}_r = \{E_B\}$ and $\mathscr{E}_b = \{E_1,E_C\}$. Hence, $x_B \in \mathcal{I}(N_2)$.

\begin{figure}[h]
    \captionsetup{justification=centering,margin=2cm}
    \begin{center}
    \begin{tikzpicture}[scale=.75]
\tikzstyle{vertex}=[inner sep=0mm]
    \node (n) at (-7,-1.5) {};
    \node (a1) at (0,0) [vertex] {};
    \node (a2) at (3,-3) [vertex] {};
    \node (a3) at (-1.5,-1.5) [vertex] {};
    \node (b1) at (1.5,0) [vertex] {};
    \node (b2) at (3,-1.5) [vertex] {};
    \node (c) at (-1.5,-3) [vertex] {};
    \node (m) at (.75,-2.25) {};
    
    \begin{scope}[fill opacity=.5]
    \filldraw[fill = red!75]
        ($(b1)+(.2,.2)$) 
        to[out=-45,in=135] ($(b2)+(.2,.2)$)
        to[out=-45,in=45] ($(b2)+(.2,-.2)$)
        to[out=-135,in=45] ($(m)+(.15,-.15)$)
        to[out=-135,in=45] ($(c)+(.2,-.1)$)
        to[out=-135,in=-45] ($(c)+(-.2,-.1)$)
        to[out=135,in=-135] ($(c)+(-.2,.3)$)
        to[out=45,in=-115] ($(m)+(-.25,.25)$)
        to[out=70,in=-30] ($(a1)+(0,-.3)$)
        to[out=150,in=270] ($(a1)+(-.2,0)$)
        to[out=90,in=180] ($(a1)+(0,.2)$)
        to[out=0,in=180] ($(b1)+(-.1,.3)$)
        to[out=0,in=135] ($(b1)+(.2,.2)$);
        
    \filldraw[fill = yellow!75]
        ($(a1)+(.2,.2)$)
        to[out=-45,in=135] ($(m)+(.15,.15)$)
        to[out=-45,in=135] ($(a2)+(.2,.2)$)
        to[out=-45,in=45] ($(a2)+(.2,-.2)$)
        to[out=-135,in=-45] ($(a2)+(-.2,-.2)$)
        to[out=135,in=-45] ($(m)+(-.15,-.15)$)
        to[out=135,in=-45] ($(a3)+(-.2,-.2)$)
        to[out=135,in=-135] ($(a3)+(-.2,.2)$)
        to[out=45,in=-135] ($(a1)+(-.2,.2)$)
        to[out=45,in=135] ($(a1)+(.2,.2)$);
        
    \filldraw[fill = blue!50]
        ($(c)+(.1,.2)$) 
        to[out=0,in=90] ($(c)+(.4,-.1)$)
        to[out=270,in=0] ($(c)+(.1,-.4)$)
        to[out=180,in=270] ($(c)+(-.2,-.1)$)
        to[out=90,in=180] ($(c)+(.1,.2)$);
    \end{scope}
    
    \draw[line width=.75pt,color=blue] (a1) -- (b1);
    \draw[line width=.75pt] (b2) -- (a2);
    \draw[line width=.75pt] (a3) -- (c);
    
    \fill (a1) circle (.075);
    \fill (a2) circle (.075);
    \fill (a3) circle (.075);
    \fill (b1) circle (.075);
    \fill (b2) circle (.075);
    \fill (c) circle (.075);
    
    \fill (a1) node [left,yshift=.5cm] {$\scaleobj{.75}{A+B}$};
    \fill (a2) node [below,xshift=.5cm] {$\scaleobj{.75}{A}$};
    \fill (a3) node [above,xshift=-.5cm] {$\scaleobj{.75}{A}$};
    \fill (b1) node [right,yshift=.5cm] {$\scaleobj{.75}{B}$};
    \fill (b2) node [above,xshift=.5cm] {$\scaleobj{.75}{B}$};
    \fill (c) node [below,xshift=-1cm] {$\scaleobj{.75}{B+C}$};
    \fill (n) node [scale=.85] {
    $\begin{tikzcd}[column sep = small, row sep = normal] 
    A \ar[d,"\kappa_3'"'] & B \ar[l,"\kappa_2"']\\
    B+C & A+B \ar[u,"\kappa_1'"]
    \end{tikzcd}$
    };
    \fill ($(m)+(-1.25,1.25)$) node {$\scaleobj{.75}{E_A}$};
    \fill ($(m)+(1.25,1.25)$) node {$\scaleobj{.75}{E_B}$};
    \fill ($(c)+(1,-.25)$) node {$\scaleobj{.75}{E_C}$};
    \fill ($(c)+(-.3,.7)$) node {$\scaleobj{.75}{E_3}$};
    \fill ($(a2)+(.3,.7)$) node {$\scaleobj{.75}{E_2}$};
    \fill ($(a1)+(.7,1)$) node {$\scaleobj{.75}{E_1}$};
    \fill ($(a1)+(-9,0)$) node {$\scaleobj{.75}{N_2:}$};
    \fill ($(a1)+(-3,0)$) node {$\scaleobj{.75}{\mathcal{H}_{N_2}:}$};
    
    \draw ($(c)+(.25,-.2)$) -- ($(c)+(.75,-.2)$);
    \draw ($(a1)+(.75,.05)$) -- ($(a1)+(.75,.65)$);
\end{tikzpicture}
    \end{center}
    \caption{A network $N_2$, obtained from $N_1$ in Figure \ref{fig:network1} by adding $B$ to the reactant complex of the reaction $A \to B$, and its network hypergraph.}
    \label{fig:network2}
\end{figure}
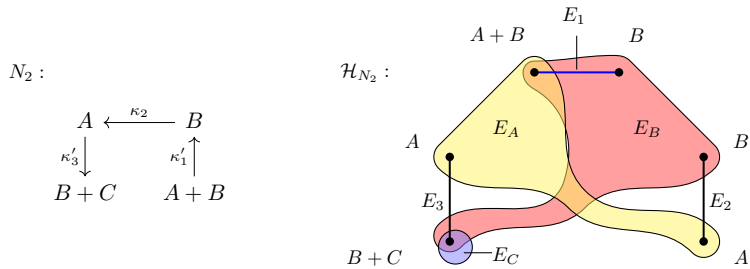
\end{example}

\begin{theorem}\label{addreaction} (Adding reactions)
Let $N$ be a 0,1-network. Suppose there is some vertex $v_i \in V(\mathcal{H}_N)$ such that $v_i$ is almost balanced with respect to the 2-colored multiset $\mathscr{E} = \mathscr{E}_r \sqcup \mathscr{E}_b$ over the edges in $E(\mathcal{H}_N)$. Let $N' = N \cup \{y_i \to \varnothing\}$. If $E_i \not \in \mathscr{E}$, then $\mathcal{I}(N') = \mathcal{I}(N)$.
\end{theorem}

\begin{proof}
Let $N = (\mathscr{S}, \mathscr{C}, \mathscr{R})$ be a 0,1-network and suppose $v_i \in V(\mathcal{H}_N)$ is almost balanced with respect to a 2-colored multiset $\mathscr{E} = \mathscr{E}_r \sqcup \mathscr{E}_b$ over the edges in $E(\mathcal{H}_N)$ such that $E_i \not \in \mathscr{E}$. Obtain a network $N' = (\mathscr{S}', \mathscr{C}', \mathscr{R}')$ from $N$ by adding the reaction $\begin{tikzcd}[column sep = small] y_i \ar[r,"\kappa_i'"] & \varnothing \end{tikzcd}$. Since $N'$ and $N$ contain the same nonempty complexes and $N$ is a 0,1-network, then $N'$ must also be a 0,1-network.

Let $S_1' = \supp{y_i}$ and $S_2' = \mathscr{S} \setminus S_1'$. If $s \in S_2'$, then $s$ is not involved in the reaction $y_i \to \varnothing$. Necessarily, the rate of change of the concentration of species $s$ in $N$ is the same as the rate of change of its concentration in $N'$. That is, $\dot{x}_s' = \dot{x}_s$ for any $s \in S_2'$. Thus, $\dot{x}_s' \in \mathcal{I}(N)$ and $\dot{x}_s \in \mathcal{I}(N')$ for all $s \in S_2'$.

Now, if $s \in S_1'$, then $\dot{x}_s'$ and $\dot{x}_s$ differ by the rate at which the concentration of $s$ changes by the reaction $y_i \to \varnothing$. More specifically, $\dot{x}_s' = \dot{x}_s - \kappa_i'x^{y_i}$. Since $v_i$ is almost balanced in $\mathcal{H}_N$, then $x^{y_i} \in \mathcal{I}(N)$ by Corollary \ref{monomial_in_ideal}. Thus, $\dot{x}_s'= \dot{x}_s - \kappa_i'x^{y_i} \in \mathcal{I}(N)$ for any $s \in S_1'$. 

Notice that the vertex set of $\mathcal{H}_{N'}$ is $V(\mathcal{H}) \cup \{u,v\}$ where $u$ and $v$ correspond to the complexes $y_i$ and $\varnothing$ in the reaction $y_i \to \varnothing$, respectively. Moreover, there are two differences between $E(\mathcal{H}_N)$ and $E(\mathcal{H}_{N'})$: first, $E(\mathcal{H}_{N'})$ contains the empty reaction edge corresponding to $y_i \to \varnothing$ and, second, if $s \in \supp{y_i}$, then $E_s$ is extended to include $u$. We will show that the 2-colored multiset $\mathscr{E} = \mathscr{E}_r \sqcup \mathscr{E}_b$ extends to a 2-colored multiset $\mathscr{E}' = \mathscr{E}_r' \sqcup \mathscr{E}_b'$ over the edges in $E(\mathcal{H}_{N'})$ such that $v_i$ is almost balanced in $\mathcal{H}_{N'}$ with respect to $\mathscr{E}'$. By the previous observations on the differences between $E(\mathcal{H}_N)$ and $E(\mathcal{H}_{N'})$, we need only show that $\mathscr{E}$ extends to $\mathscr{E}'$ in such a way that $\deg_{\mathscr{E}_r'}(u) = \deg_{\mathscr{E}_b'}(u)$ and $\deg_{\mathscr{E}_r'}(v) = \deg_{\mathscr{E}_b'}(v)$. Since $v$ corresponds to the empty complex, then $v$ is isolated and, hence, $\deg_{\mathscr{E}_r'}(v) = \deg_{\mathscr{E}_b'}(v)$, regardless of the choice of $\mathscr{E}'$.

Let $\mathscr{E}'=\mathscr{E}_r' \sqcup \mathscr{E}_b'$ where $\mathscr{E}_r' = \mathscr{E}_r$ but viewed as a multiset over the edges in $E(\mathcal{H}_{N'})$ and $\mathscr{E}_b' = \mathscr{E}_b$ but viewed as a multiset over the edges in $E(\mathcal{H}_{N'})$. Since $v_i$ is almost balanced with respect to $\mathscr{E} = \mathscr{E}_r \sqcup \mathscr{E}_b$, then it is almost balanced with respect to $\mathscr{E}'=\mathscr{E}_r' \sqcup \mathscr{E}_b'$, in particular, $\deg_{\mathscr{E}_r'}(u_i) = \deg_{\mathscr{E}_b'}(u_i)$. Since $E_i \not \in \mathscr{E}$, it is the case $E_i \notin \mathscr{E}'$, so $\deg_{\mathscr{E}_r'}(u_i) = \deg_{\mathscr{E}_b'}(u_i)$ solely counts the number of red and blue species edges that cover $u_i$. Note that both $u_i$ and $u$ correspond to the complex $y_i$. Then the number of red and blue species edges that cover $u$ is the same as the number of red and blue species edges that cover $u_i$. That is, $\deg_{\mathscr{E}_r'}(u) = \deg_{\mathscr{E}_b'}(u)$. Thus, $v_i$ is almost balanced in $\mathcal{H}_{N'}$ with respect to $\mathscr{E}' = \mathscr{E}_r' \sqcup \mathscr{E}_b'$. By Corollary \ref{monomial_in_ideal}, $x^{y_i} \in \mathcal{I}(N')$ and so $\dot{x}_s = \dot{x}_s' + \kappa_i'x^{y_i} \in \mathcal{I}(N')$ for any $s \in S_1'$. Therefore, $\mathcal{I}(N) = \mathcal{I}(N')$.
\end{proof}

\begin{example}
In the hypergraph $\mathcal{H}_{N_2}$ in Figure \ref{fig:network2}, we see that $v_3$, the vertex labeled $B+C$, is almost balanced with respect to $\mathscr{E} = \mathscr{E}_r \sqcup \mathscr{E}_b$ where $\mathscr{E}_r = \{E_C\}$ and $\mathscr{E}_b = \varnothing$. Since $E_3 \not \in \mathscr{E}$, then the network $N_3$ obtained from $N_2$ by adding the reaction $A \to \varnothing$ has the same steady state ideal as $N_2$; see Figure 
\ref{fig:network3}. Further, $v_3$ remains almost balanced in $\mathcal{H}_{N_3}$ with respect to its previous 2-coloring of $\mathscr{E}$ so $x_A \in \mathcal{I}(N_3)$. To see $x_B \in \mathcal{I}(N_3)$, notice that $u_2 \in V(\mathcal{H}_{N_3})$ is almost balanced for the same reason it is almost balanced in Example \ref{example2}.

\begin{figure}[h]
    \centering
    \captionsetup{justification=centering,margin=2cm}
    \begin{tikzpicture}[scale=.75]
\tikzstyle{vertex}=[inner sep=0mm]
    \node (n) at (-7,-1.5) {};
    \node (a1) at (0,0) [vertex] {};
    \node (a2) at (3,-3) [vertex] {};
    \node (a3) at (-1.5,-1.5) [vertex] {};
    \node (b1) at (1.5,0) [vertex] {};
    \node (b2) at (3,-1.5) [vertex] {};
    \node (c) at (-1.5,-3) [vertex] {};
    \node (m) at (.75,-2.25) {};
    \node (a4) at (1.5,-4.5) [vertex] {};
    \node (v) at (0,-4.5) [vertex] {};
    
    \begin{scope}[fill opacity=.5]
    \filldraw[fill = orange!75]
        ($(b1)+(.2,.2)$) 
        to[out=-45,in=135] ($(b2)+(.2,.2)$)
        to[out=-45,in=45] ($(b2)+(.2,-.2)$)
        to[out=-135,in=45] ($(m)+(.15,-.15)$)
        to[out=-135,in=45] ($(c)+(.2,-.1)$)
        to[out=-135,in=-45] ($(c)+(-.2,-.1)$)
        to[out=135,in=-135] ($(c)+(-.2,.3)$)
        to[out=45,in=-115] ($(m)+(-.25,.25)$)
        to[out=70,in=-30] ($(a1)+(0,-.3)$)
        to[out=150,in=270] ($(a1)+(-.2,0)$)
        to[out=90,in=180] ($(a1)+(0,.2)$)
        to[out=0,in=180] ($(b1)+(-.1,.3)$)
        to[out=0,in=135] ($(b1)+(.2,.2)$);
        
    \filldraw[fill = yellow!75]
        ($(a1)+(.2,.2)$)
        to[out=-45,in=135] ($(m)+(.15,.15)$)
        to[out=-45,in=135] ($(a2)+(.2,.2)$)
        to[out=-45,in=45] ($(a2)+(.2,-.2)$)
        to[out=-135,in=45] ($(a4)+(.2,-.2)$)
        to[out=-135,in=-45] ($(a4)+(-.2,-.2)$)
        to[out=135,in=-45] ($(m)+(-.15,-.15)$)
        to[out=135,in=-45] ($(a3)+(-.2,-.2)$)
        to[out=135,in=-135] ($(a3)+(-.2,.2)$)
        to[out=45,in=-135] ($(a1)+(-.2,.2)$)
        to[out=45,in=135] ($(a1)+(.2,.2)$);
        
    \filldraw[fill = red!75]
        ($(c)+(.1,.2)$) 
        to[out=0,in=90] ($(c)+(.4,-.1)$)
        to[out=270,in=0] ($(c)+(.1,-.4)$)
        to[out=180,in=270] ($(c)+(-.2,-.1)$)
        to[out=90,in=180] ($(c)+(.1,.2)$);
    \end{scope}
    
    \draw[line width=.75pt] (a1) -- (b1);
    \draw[line width=.75pt] (b2) -- (a2);
    \draw[line width=.75pt] (a3) -- (c);
    
    \fill (a1) circle (.075);
    \fill (a2) circle (.075);
    \fill (a3) circle (.075);
    \fill (b1) circle (.075);
    \fill (b2) circle (.075);
    \fill (c) circle (.075);
    \fill (a4) circle (.075);
    \fill (v) circle (.075);
     
    \fill (a1) node [left,yshift=.5cm] {$\scaleobj{.75}{A+B}$};
    \fill (a2) node [below,xshift=.5cm] {$\scaleobj{.75}{A}$};
    \fill (a3) node [above,xshift=-.5cm] {$\scaleobj{.75}{A}$};
    \fill (b1) node [right,yshift=.5cm] {$\scaleobj{.75}{B}$};
    \fill (b2) node [above,xshift=.5cm] {$\scaleobj{.75}{B}$};
    \fill (c) node [below,xshift=-1cm] {$\scaleobj{.75}{B+C}$};
    \fill (a4) node [right,yshift=-.5cm] {$\scaleobj{.75}{A}$};
    \fill (v) node [left,yshift=-.5cm] {$\scaleobj{.75}{\varnothing}$};
    
    \fill ($(n)+(0,-.5)$) node [scale=.85] {
    $\begin{tikzcd}[column sep = small, row sep = normal] 
    \varnothing & A \ar[l,"\kappa_4"'] \ar[d,"\kappa_3'"'] & B \ar[l,"\kappa_2"']\\
    & B+C & A+B \ar[u,"\kappa_1'"]
    \end{tikzcd}$
    };
    \fill (-.35,-.925) node {$\scaleobj{.75}{E_A}$};
    \fill (1.5,-.925) node {$\scaleobj{.75}{E_B}$};
    \fill ($(c)+(1,-.25)$) node {$\scaleobj{.75}{E_C}$};
    \fill ($(c)+(-.3,.7)$) node {$\scaleobj{.75}{E_3}$};
    \fill ($(a2)+(.3,.7)$) node {$\scaleobj{.75}{E_2}$};
    \fill ($(a1)+(.7,1)$) node {$\scaleobj{.75}{E_1}$};
    \fill ($(n)+(-3,1)$) node {$\scaleobj{.75}{N_3:}$};
    \fill ($(a1)+(-3,0)$) node {$\scaleobj{.75}{\mathcal{H}_{N_3}:}$};
    
    \draw ($(c)+(.25,-.2)$) -- ($(c)+(.75,-.2)$);
    \draw ($(a1)+(.75,.05)$) -- ($(a1)+(.75,.65)$);
\end{tikzpicture}
    \caption{A network $N_3$, obtained from $N_2$ in Figure \ref{fig:network2} by adding the reaction $A \to \varnothing$, and its network hypergraph.}
    \label{fig:network3}
\end{figure}
\end{example}

\begin{example}\label{diff_networks_same_ideal}
Let us see an example of networks with the same steady-state ideal but where we cannot move between them using the above three operations. Consider the following two networks:
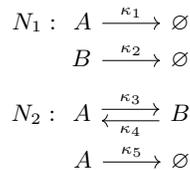
\[
\begin{array}{rl}
    N_1: \begin{tikzcd}
    A  \ar[r,"\kappa_1"] & \varnothing
    \end{tikzcd} &  \\
    \begin{tikzcd}
    B \ar[r,"\kappa_2"] & \varnothing
    \end{tikzcd}
\end{array}
\]
\[
\begin{array}{rl}
     N_2: \begin{tikzcd}[every arrow/.append style={shift left}]
     A \ar[r,"\kappa_3"] & B \ar[l,"\kappa_4"]
     \end{tikzcd} &\\
     \begin{tikzcd}
     A \ar[r,"\kappa_5"] & \varnothing
     \end{tikzcd}&
\end{array}
\]
Both reaction networks $N_1$ and $N_2$ have $\langle x_A,x_B \rangle$ as their steady-state ideal. From a quick inspection, we might believe that these networks can be obtained by the other via a sequence of operations. However, that is not that case. In short, the reason is because $N_2$ is a minimal network with respect to the ideal preserving operations. Hence, if $N_2$ were obtained from $N_1$ by a sequence of operations, it would contradict the minimality of $N_2$ since $N_1$ is the ``smaller'' network. Thus, we need only show that $N_2$ is indeed minimal. Thinking about the operations in reverse, we can either remove a species or remove a reaction of the form $y \to \varnothing$. Notice that removing any species is out of the picture since each nonempty reactant complex in $N_2$ has exactly one species (with coefficient 1), that is, $N_2$ is \textit{monomolecular}. Thus, all we have left is to possibly remove any reactions. Removing a reaction requires a product complex to be almost balanced. By Proposition \ref{reversible_is_balanced}, the presence of the reversible reactions prevents the vertices $u_3$, $v_3$, $u_4$, and $v_4$ from being almost balanced. Thus, we cannot remove the fifth reaction $A \to \varnothing$. Also, we cannot remove $A \to B$ since that would first require removing $B$ from the product, which a priori we cannot do. The reaction $B \to A$ cannot be removed for the same reason. We are now all out of operations so $N_2$ must be minimal.

\end{example}

\section{Discussion}

In this paper, we discuss three ideal preserving operations on reaction networks based on the underling combinatorics of their associated network hypergraphs.  Each of these operations require the existence of an almost balanced vertex in the network hypergraph and thus an almost balanced edge set.  We note that determining the existence of an almost balanced vertex is a difficult problem to solve in general. Indeed, the authors of \cite{PS} show that balanced edge sets, i.e. hypergraphs with possible multiple edges whose edges can be 2-colored so that every vertex is covered by the same number of red and blue edges, have duals that are necessarily discrepancy zero. A result of \cite{CNN} shows that it is NP-hard to distinguish if a hypergraph has discrepancy zero or if there is a vertex with discrepancy bounded below by $\Omega(\sqrt{m})$ where $m$ is the number of edges. Hypergraphs that satisfy the conditions of our main theorems that contain almost balanced vertices have subhypergraphs with small discrepancy, thus we expect algorithmically checking for almost-balanced vertices to be troublesome.  Although, while determining the existence of a balanced vertex may be hard in general, we do give several conditions that guarantees that a specific vertex cannot be almost balanced: (1) the vertex corresponds to a complex in a reversible reaction and (2) the vertex corresponds to a complex in the a cycle in the reaction graph. 

One motivation for this paper is the classification of steady-state ideals.  As a first step to this problem, we can ask: \emph{when is the steady-state ideal monomial?}  Our results give a partial answer to this question. Specifically, Theorem \ref{min_react_generate_ideal} states that, for 0,1-networks, the steady-state ideal is monomial if every reaction contains an almost balanced vertex.  We do not yet have a necessary condition that can test whether a steady-state ideal is monomial. As we see in network $N_2$ from  Example \ref{diff_networks_same_ideal}, there can be situations where a monomial is contained in the steady-state ideal, but its corresponding complex does not belong to a reaction with an almost balanced vertex.  In Example \ref{diff_networks_same_ideal}, this happens because the reaction is reversible.  Thus, perhaps a next direction to this work is to understand how to handle reversible reactions and, more generally, how to handle reaction cycles within the reaction network. We suspect there is a hidden fourth operation at play here involving reaction cycles. For example, we would like our operations to detect that the networks $\begin{tikzcd}[column sep = small] \varnothing & A \ar[r,shift left] \ar[l] & B \ar[l,shift left] \end{tikzcd}$ and $\begin{tikzcd}[column sep = small] A \ar[r] & \varnothing & B \ar[l] \end{tikzcd}$ have the same steady state ideal, that is, obtain one network from the other via a sequence of operations. Such an operation would be quite natural as we would expect the former network to behave as the latter due to the dependence of $A$ and $B$ caused by the reaction cycle. 

Another, perhaps more difficult, direction to take is to generalize the definition of almost balanced to handle reaction networks which are not 0,1-networks. This requires one to be able to ``see'' the coefficients of stoichiometry. The information encoded by a network hypergraph merely captures the presence of either a reaction or a species within a reaction regardless of its stoichiometric coefficient. Namely, the network hypergraph merely decides if $\gamma_{s_i} = 0$ or $\gamma_{s_i} \neq 0$. To combat this, we might include the stoichiometric coefficients in the definition of almost balanced. We might propose to revise the equations in the current definition of almost balanced to 
\begin{align*}
    \sum_{e \in \mathscr{E}_r} \gamma(v,e) &\neq \sum_{e \in \mathscr{E}_b} \gamma(v,e) \\
    \sum_{e \in \mathscr{E}_r}\gamma(u,e) &= \sum_{e \in \mathscr{E}_b}\gamma(u,e) \quad , \; u \neq v 
\end{align*}
where $\gamma$ is the map $V \times E \to \mathbb{Z}$ defined by
\[
\gamma(v,e) = 
\begin{cases}
0 &\text{ if } v \not\in e\\
1 &\text{ if } e = E_i \\
\gamma_{s_i} &\text{ if } v \in E_i, e = E_s
\end{cases}
\]
Notice that if $\gamma_{s_i} = 0$ or $1$ then the above equations are equivalent to the current definition of almost balanced.

 Finally, we are very interested in discovering new operations that preserve the steady-state ideal. The operations we describe can only be applied if we have first found an almost balance vertex. Hence, the operations work on networks with at least one monomial in their steady state ideal. However, most biologically relevant networks do not have any monomials in their steady state ideal. In those cases, we would like to have ideal preserving operations that can reduce a network down to its simplest form while still maintaining its algebraic structure. This might require an analogue of almost balanced to recover binomials or, more generally, polynomials with a fixed number of terms. 

\section{Acknowledgements}
Elizabeth Gross was supported by the National Science Foundation DMS-1620109 and DMS-1945584. We thank Jordan Schettler and Matthew Johnston for feedback on an earlier draft of this manuscript.
\bibliographystyle{plain} 
\bibliography{references} 

\begin{thebibliography}{10}

\bibitem{boros}
Balázs Boros.
\newblock Existence of positive steady states for weakly reversible mass-action
  systems, 2017.

\bibitem{MDSC}
Mercedes P\'erez~Mill\'an Carston~Conradi, Alicia~Dickenstein and Anne Shiu.
\newblock Chemical reaction systems with toric steady states.
\newblock {\em Bulletin of Mathematical Biology}, 74:1027--1065, 2012.

\bibitem{CNN}
Moses Charikar, Alantha Newman, and Aleksandar Nikolov.
\newblock Tight hardness results for minimizing discrepancy.
\newblock In {\em Proceedings of the Twenty-Second Annual ACM-SIAM Symposium on
  Discrete Algorithms}, SODA ’11, page 1607–1614, USA, 2011. Society for
  Industrial and Applied Mathematics.

\bibitem{toricdynamicalsystems}
Gheorghe Craciun, Alicia Dickenstein, Anne Shiu, and Bernd Sturmfels.
\newblock Toric dynamical systems.
\newblock {\em Journal of Symbolic Computation}, 44(11):1551--1565, 2009.

\bibitem{CJY}
Gheorghe Craciun, Jiaxin Jin, and Polly Yu.
\newblock An efficient characterization of complex balanced, detailed balanced,
  weakly reversible, and reversible systems, 12 2018.

\bibitem{DickensteinSurvey}
Alicia Dickenstein.
\newblock Biochemical reaction networks: An invitation for algebraic geometers.
\newblock In {\em Mathematical congress of the Americas}, volume 656, pages
  65--83. Contemp. Math, 2016.

\bibitem{feinberg1972}
Martin Feinberg.
\newblock Complex balancing in general kinetic systems.
\newblock {\em Archive for rational mechanics and analysis}, 49(3):187--194,
  1972.

\bibitem{modelselection}
Elizabeth Gross, Brent Davis, Kenneth~L. Ho, Daniel~J. Bates, and Heather~A.
  Harrington.
\newblock Numerical algebraic geometry for model selection and its application
  to the life sciences.
\newblock {\em Journal of The Royal Society Interface}, 13(123):20160256, 2016.

\bibitem{Joining}
Elizabeth Gross, Heather Harrington, Nicolette Meshkat, and Anne Shiu.
\newblock Joining and decomposing reaction networks.
\newblock {\em Journal of Mathematical Biology}, pages 1--49, 2020.

\bibitem{GHRS}
Elizabeth Gross, Heather~A Harrington, Zvi Rosen, and Bernd Sturmfels.
\newblock Algebraic systems biology: a case study for the wnt pathway.
\newblock {\em Bulletin of Mathematical Biology}, 78(1):21--51, 2016.

\bibitem{gross2013combinatorial}
Elizabeth Gross and Sonja Petrovi{\'c}.
\newblock Combinatorial degree bound for toric ideals of hypergraphs.
\newblock {\em International Journal of Algebra and Computation},
  23(06):1503--1520, 2013.

\bibitem{Gunawardena}
Jeremy Gunawardena.
\newblock Distributivity and processivity in multisite phosphorylation can be
  distinguished through steady-state invariants.
\newblock {\em Biophysical Journal}, 93(11):3828 -- 3834, 2007.

\bibitem{HarringtonHo}
Heather~A Harrington, Kenneth~L Ho, Thomas Thorne, and Michael~PH Stumpf.
\newblock Parameter-free model discrimination criterion based on steady-state
  coplanarity.
\newblock {\em Proceedings of the National Academy of Sciences},
  109(39):15746--15751, 2012.

\bibitem{horn1972}
Fritz Horn and Roy Jackson.
\newblock General mass action kinetics.
\newblock {\em Archive for rational mechanics and analysis}, 47(2):81--116,
  1972.

\bibitem{maclean2015parameter}
Adam~L MacLean, Zvi Rosen, Helen~M Byrne, and Heather~A Harrington.
\newblock Parameter-free methods distinguish wnt pathway models and guide
  design of experiments.
\newblock {\em Proceedings of the National Academy of Sciences},
  112(9):2652--2657, 2015.

\bibitem{PS}
Sonja Petrovi{\'c} and Despina Stasi.
\newblock Toric algebra of hypergraphs.
\newblock {\em Journal of Algebraic Combinatorics}, 39(1):187--208, 2014.

\bibitem{SS}
Anne Shiu and Bernd Sturmfels.
\newblock Siphons in chemical reaction networks.
\newblock {\em Bulletin of Mathematical Biology}, 72(6):1448--1463, 2010.

\end{thebibliography}

\end{document}